\documentclass[11pt]{article}
\usepackage{amsmath,amssymb,amsthm}

\usepackage{multicol}

\usepackage{color}

\usepackage[latin1]{inputenc}

{\catcode `\@=11 \global\let\AddToReset=\@addtoreset}
\AddToReset{equation}{section}

\def\mathsf{\bf}
\def\N{\mathbb{N}}
\def\R{\mathbb{R}}
\def\Z{\mathbb{Z}}

\def\d{\mathrm d}

\def\la{\lambda}
\def\ga{\gamma}
\def\vep{\varepsilon}

\def\E{\mathrm E}

\def\1{{\bf 1}}

\newcommand{\nn}{\nonumber}
\newcommand{\noi}{\noindent}

\newcommand{\mba}{\boldsymbol{a}}
\newcommand{\mbb}{\boldsymbol{b}}
\newcommand{\mbc}{\boldsymbol{c}}

\newcommand{\mbt}{\boldsymbol{t}}
\newcommand{\mbs}{\boldsymbol{s}}
\newcommand{\mbu}{\boldsymbol{u}}

\newcommand{\mbx}{\boldsymbol{x}}
\newcommand{\mby}{\boldsymbol{y}}
\newcommand{\mbz}{\boldsymbol{z}}

\newcommand{\mbgamma}{\boldsymbol{\gamma}}

\def\limd{\renewcommand{\arraystretch}{0.5}
\begin{array}[t]{c}
\stackrel{\rm d}{\longrightarrow} \\
\end{array}\renewcommand{\arraystretch}{1}}

\def\limfdd{\renewcommand{\arraystretch}{0.5}
\begin{array}[t]{c}
\stackrel{\rm fdd}{\longrightarrow} \\
\end{array}\renewcommand{\arraystretch}{1}}

\def\eqfdd{\renewcommand{\arraystretch}{0.5}
\begin{array}[t]{c}
\stackrel{\rm fdd}{=} \\
\end{array}\renewcommand{\arraystretch}{1}}

\def\neqfdd{\renewcommand{\arraystretch}{0.5}
\begin{array}[t]{c}
\stackrel{\rm fdd}{\neq} \\
\end{array}\renewcommand{\arraystretch}{1}}

\newtheorem{thm}{Theorem}[section]
\newtheorem{cor}[thm]{Corollary}

\newtheorem{prop}[thm]{Proposition}
\newtheorem{defn}[thm]{Definition}
\newtheorem{rem}{Remark}[section]

\begin{document}

\title{Scaling limits of
linear random fields on ${\Z}^2$ \\ with general
dependence axis }

\author{
Vytaut\.e Pilipauskait\.e$^1$  \ and \
Donatas Surgailis$^2$}
\date{\today \\ \small
	\vskip.2cm
$^1$Aarhus University, Department of Mathematics, Ny Munkegade 118, 8000 Aarhus C, Denmark\\
$^2$Vilnius University, Faculty of Mathematics and Informatics, Naugarduko 24, 03225 Vilnius, Lithuania
}
\maketitle

\begin{quote}

{\bf Abstract.}  We discuss anisotropic scaling of long-range dependent linear random fields 
$X$ on $\Z^2$ with arbitrary dependence axis
(direction in the plane along which the moving-average coefficients decay at a smallest rate).
The scaling limits are taken over rectangles  whose sides are parallel to the coordinate axes and
increase as $\lambda $ and $\lambda^\gamma $ when $\lambda \to \infty$, for any $\gamma >0$. The scaling  transition 
occurs
at $\gamma^X_0 >0$ if the scaling limits of $X$ are different and do not depend on $\gamma $ for $\gamma > \gamma^X_0 $ and $\gamma < \gamma^X_0$.
We prove that the fact of `oblique' dependence axis (or incongruous scaling) dramatically changes the
scaling transition in the above model
so that 
$\gamma_0^X = 1$  independently of other parameters, contrasting  the results in
Pilipauskait\.e  and Surgailis (2017) on the scaling transition under
congruous scaling.

\end{quote}

\smallskip
{\small

\noi {\it Keywords:} random field; long-range dependence; dependence axis; anisotropic scaling limits; scaling  transition;
fractional Brownian sheet }

\vskip-1cm

\section{Introduction}

\cite{dam2017, dam2019, pils2014, pils2016, pils2017, ps2015,  ps2016, sur2019a, sur2019b}
discussed scaling limits
\begin{equation}\label{scale}
\big\{ A^{-1}_{\la, \mbgamma} S^X_{\la, {\mbgamma}}({\mbx}), \, \mbx \in \R^\nu_+ \big\}  \limfdd V^X_{\mbgamma},
\quad \lambda \to \infty,
\end{equation}
for some classes of stationary random fields  (RFs) $X = \{X({\mbt}), \, {\mbt}  \in \Z^\nu \}$, 
where $A_{\la, \mbgamma} \to \infty$ 
is a normalization and 
\begin{equation}\label{SXsum}
S^X_{\la,{\mbgamma}}({\mbx}) := \sum_{{\mbt} \in K_{\lambda,\mbgamma}(\mbx) 
} X({\mbt}), \quad \mbx \in \R^\nu_+,
\end{equation}
are partial sums of RF $X$ over rectangles $
K_{\lambda,\mbgamma}(\mbx) := \{{\mbt} =  (t_1, \dots, t_\nu)^\top \in \Z^\nu: 0 < t_i \le \lambda^{\gamma_i} x_i, \,
i=1, \dots, \nu \}$
and ${\mbgamma} = (\gamma_1, \dots, \gamma_\nu)^\top \in \R^\nu_+$ is {\it arbitrary}.
Following \cite{pils2016, sur2019a}  the family
$\{ V^X_{\mbgamma}, \, {\mbgamma} \in \R^\nu_+\} $ of all scaling limits in \eqref{scale} will
be called the {\it scaling diagram of RF $X$}.
Recall that a stationary RF $X$ with $\operatorname{Var}(X(\boldsymbol{0}))<\infty$ is said long-range dependent (LRD) if $\sum_{\boldsymbol{t} \in \Z^\nu} |\operatorname{Cov}(X(\boldsymbol{0}),X(\boldsymbol{t}))| = \infty$, see  \cite{lah2016, ps2016}. 
\cite{pils2017, ps2015, ps2016}
observed that for a large class of 
LRD RFs  $X$ in dimension $\nu =2$, the scaling
diagram essentially consists of three points.
More precisely  (assuming $\gamma_1 = 1, \gamma_2 = \gamma$ w.l.g.),
there exists a (nonrandom) $\gamma^X_0 >0$ such that $V^X_\gamma \equiv V^X_{(1,\gamma)} $ do not depend on $\gamma $ for  $\gamma > \gamma^X_0$ and
$\gamma < \gamma^X_0$, viz.,
\begin{equation}\label{Vlim}
V^X_\gamma = \begin{cases} V^X_+,  &\gamma > \gamma^X_0, \\
V^X_-, &\gamma < \gamma^X_0, \\
V^X_0, &\gamma = \gamma^X_0,
\end{cases}
\end{equation}
and $V^X_+ \neqfdd a V^X_-$, $\forall a >0$.  The above fact was termed the {\it scaling transition} \cite{ps2015, ps2016},
$V^X_0$ called the {\it well-balanced} and $V^X_\pm $ the {\it unbalanced} scaling limits of $X$. In the sequel, we  shall also refer to
$\gamma^X_0 >0$ in \eqref{Vlim} as the {\it scaling transition} or the  {\it critical point}. The
 existence of the scaling transition was established for a wide class of planar linear and nonlinear RF models
 including those appearing in telecommunications and econometrics. See the review paper \cite{sur2019c} for
 further discussion and recent developments.

Particularly, \cite{pils2017} discussed anisotropic scaling of
linear LRD RFs $X$ on $\Z^2$
written as a moving-average
\begin{equation}\label{Xlin}
X({\mbt}) = \sum_{{\mbs} \in \Z^2} a({\mbt}-{\mbs}) \vep({\mbs}), \quad {\mbt} \in \Z^2,
\end{equation}
of standardized i.i.d.\ sequence $\{ \vep({\mbt}), \, {\mbt} \in \Z^2\} $ 
with deterministic 
coefficients
\begin{equation}\label{acoefL}
a({\mbt}) = \frac{1}{|t_1|^{q_1} +  |t_2|^{q_2}}\Big(L_{{\rm sign}(t_2)}
\big(\frac{t_1}{(|t_1|^{q_1} +|t_2|^{q_2})^{1/q_1}}\big) + o(1)\Big),
\quad |{\mbt}| := |t_1|+|t_2| \to  \infty,
\end{equation}
where $q_i>0$, $i=1,2$, satisfy
\begin{equation}\label{qLRD}
1 < Q := \frac{1}{q_1} + \frac{1}{ q_2} <  2
\end{equation}
and 
$L_\pm$ are continuous functions on $[-1,1]$, $L_+(\pm 1) = L_-(\pm 1) =: L_0 (\pm 1)$. (\cite{pils2017,sur2019b} use a slightly different form of moving-average coefficients $a$ and
assume $L_+ = L_- $ but their results are valid for $a$ in \eqref{acoefL}. See also Sec.~4 below.)
Since $a(t,0) = O(|t|^{-q_1})$, $a(0,t) = O(|t|^{-q_2})$, $t \to \infty$, for $q_1 \ne q_2$  decay at different rate in the horizontal and vertical directions,
the ratio $\gamma^0 := q_1/q_2$
can be regarded as `intrinsic (internal) scale ratio' and the exponent $\gamma >0$ as `external scale ratio',
characterizing the anisotropy of the RF $X$ in \eqref{Xlin} and the scaling procedure in \eqref{scale}--\eqref{SXsum},
respectively.  Indeed, the scaling transition for the above $X$
occurs at the point $\gamma^X_0 = \gamma^0$ where
these ratios coincide \cite{pils2017}. Let us remark that isotropic scaling of linear and nonlinear RFs on $\Z^\nu $ and $\R^\nu$ was discussed in
\cite{dob1979, dobmaj1979, leo1999, lah2016} and other works, while the scaling limits of linear random processes with one-dimensional `time' (case $\nu =1$)
were identified in \cite{dav1970}. We also refer to the monographs \cite{dou2003, ber2013, book2012} on various
probabilistic and statistical aspects of long-range dependence.

A direction in the plane (a line passing through the origin) along which the moving-average coefficients $a$ decay at the smallest rate may be
called the {\it dependence axis} of RF $X$ in \eqref{Xlin}. The rigorous definition of dependence axis is given in Sec.~4.
Due to the form in \eqref{acoefL} the dependence axis agrees with the horizontal axis if $q_1 < q_2$ and with
the vertical axis if $q_1 > q_2$, see Proposition \ref{depaxis}.
Since the scaling in \eqref{scale}--\eqref{SXsum} is parallel to the coordinate axes,
we may say that for RF $X$ in \eqref{Xlin}--\eqref{acoefL}, {\it the scaling is congruous with the dependence axis of $X$}
and the results of \cite{pils2017, sur2019b} (as well as of \cite{ps2015,ps2016})
refer to this rather specific situation. The situation when the dependence axis does not agree with any of 
the two coordinate axes (the case of {\it incongruous scaling}) seems to be  more common  and then one may naturally ask
about the scaling transition and the scaling transition point $\gamma^X_0$  under incongruous scaling.

The present paper discusses the above problem for linear RF in \eqref{Xlin} with the moving-average
coefficients
\begin{equation}\label{bcoefL}
b({\mbt}) = \frac{1}{|{\mbb}_1\cdot {\mbt}|^{q_1} +  |{\mbb}_2\cdot {\mbt}|^{q_2}}\Big(L_{\operatorname{sign}({\mbb}_2\cdot {\mbt})}
\big(\frac{{\mbb}_1\cdot {\mbt}}{(|{\mbb}_1\cdot {\mbt}|^{q_1} +|{\mbb}_2\cdot {\mbt}|^{q_2})^{1/q_1}}\big) + o(1)\Big),
\quad |{\mbt}|  \to  \infty,
\end{equation}
where ${\mbb}\cdot {\mbt} := b_1t_1 + b_2 t_2 $ is the scalar product, $\mbb_i = (b_{i1}, b_{i2})^\top$, $i=1,2$,
are real vectors, $B= (b_{ij})_{i,j=1,2}$
is a nondegenerate matrix and $q_i > 0$, $i=1,2$, $Q  \in (1,2)$,
$L_\pm $ are the same as in \eqref{acoefL}.
The dependence axis of 
$X$ with coefficients $b$ in \eqref{bcoefL} 
 is given by
\begin{equation}
{\mbb}_2\cdot {\mbt} = 0 \ \ (q_1 < q_2) \quad \text{or} \quad {\mbb}_1\cdot {\mbt} = 0 \ \ (q_1 > q_2),
\end{equation}
see Proposition \ref{depaxis} below, and generally does not agree with the coordinate axes, which results
in incongruous scaling in \eqref{scale}. We prove that the last
fact completely changes the scaling transition. Namely, under incongruous scaling
{\it the scaling transition point $\gamma^X_0$ in \eqref{Vlim} is always 1}: $\gamma^X_0 =1 $
for 
any $q_1>0$, $q_2>0 $ satisfying \eqref{qLRD}, and the unbalanced limits $V^X_\pm $ are generally different from
the corresponding limits in the congruous scaling case.
The main results of this paper
are illustrated in Table 1.  Throughout the paper we use the notation
\begin{eqnarray}\label{Hnota}
\widetilde Q_i := Q - \frac{1}{2q_i}, \quad
\tilde H_i := 1 - \frac{q_i}{2} (2-Q), \quad H_i:= \frac{1}{2} + q_i (Q-1), \quad i=1,2,
\end{eqnarray}
and $B_{{\cal H}_1, {\cal H}_2} = \{B_{{\cal H}_1, {\cal H}_2}({\mbx}), \, {\mbx} \in \R^2_+ \} $ for
fractional Brownian sheet (FBS) with Hurst parameters
$0< {\cal H}_i \le 1$, $i=1,2$, defined as a Gaussian RF with zero mean and covariance
$\E B_{{\cal H}_1, {\cal H}_2}({\mbx}) B_{{\cal H}_1, {\cal H}_2}({\mby}) = (1/4)
\prod_{i=1}^2 (x_i^{2{\cal H}_i} + y_i^{2{\cal H}_i} - |x_i-y_i|^{2{\cal H}_i})$, ${\mbx}, {\mby} \in \R^2_+.  $

\begin{table}[htbp]

\begin{center}

\begin{tabular}{| c | c | c | c | c |c| c |}
\hline
{$\begin{array}{c}
{\rm Parameter} \\
{\rm region}
\end{array}$
} &\multicolumn{3}{c |}{Congruous scaling} &\multicolumn{3}{c |}{Incongruous scaling}\\
\cline{2-7}
&{Critical $\gamma^X_0$} &{$V^X_+$} & {$V^X_-$} &{Critical $\gamma^X_0$} &{$V^X_+$} & {$V^X_-$} \\
\hline
{$\widetilde Q_1 \wedge \widetilde Q_2 > 1 $} & $q_1/q_2$ & $B_{1,\tilde H_2}$  & $B_{\tilde H_1,1}$ & $1$
& $B_{1, \tilde H_1 \wedge \tilde H_2}$  & $B_{\tilde H_1 \wedge \tilde H_2,1}$ \\
\hline
$\widetilde Q_1 < 1 < \widetilde Q_2 $ & $q_1/q_2$ & $B_{H_1, 1/2}$   &  $B_{\tilde H_1,1}$ &   $1$
& $B_{H_1\wedge H_2, 1/2}$   &  $B_{1/2, H_1\wedge H_2}$ \\
\hline
$\widetilde Q_2 < 1 < \widetilde Q_1 $ & $q_1/q_2$ & $B_{1, \tilde H_2}$  &  $B_{1/2, H_2}$   &  $1$ & $B_{H_1\wedge H_2, 1/2}$   &  $B_{1/2, H_1\wedge H_2}$    \\
\hline
$\widetilde Q_1 \vee \widetilde Q_2 < 1    $ & $q_1/q_2$ & $B_{H_1, 1/2}$   &  $B_{1/2, H_2}$ &   $1$  & $B_{H_1\wedge H_2, 1/2}$   &  $B_{1/2, H_1\wedge H_2}$  \\
\hline
\end{tabular}
\caption{Unbalanced scaling limits $V^X_\pm$  (without asymptotic constants) 
under congruous and incongruous scaling.
 }\label{tab:1}
\end{center}

\end{table}

We expect that the results of the present paper can be extended to {\it negatively dependent} linear RFs with
coefficients as in \eqref{bcoefL} satisfying $Q < 1 $ (which guarantees their summability) and the zero-sum
condition $\sum_{{\mbt} \in \Z^2} b({\mbt}) = 0$. The existence of the scaling transition for negatively dependent RFs
with coefficients as in \eqref{acoefL} (i.e., under congruous scaling) was established in \cite{sur2019b}. Let us note
that the case of negative dependence is more delicate, due to the possible occurrence  of edge effects, see \cite{lah2016,sur2019b}.
Further interesting  open problems concern incongruous scaling of {\it nonlinear } or {\it subordinated} RFs on $\Z^2 $ (see \cite{pils2017})
and possible extensions to $\Z^3 $ and higher dimensions. We mention that  the scaling diagram of linear LRD RF on $\Z^3 $ under congruous
scaling is quite complicated, see \cite{sur2019a}; the incongruous scaling may lead
to a much more simple result akin to Table 1.

Sec.~2 contains the main results (Theorems  \ref{thmR11}--\ref{thmR1}), together with rigorous assumptions and
the definitions of the limit RFs.  The proofs of these facts are given in Sec.~3. Sec.~4 (Appendix) contains the definition
and the existence of the dependence axis for moving-average coefficients $b$ as in \eqref{bcoefL} (Proposition \ref{depaxis}). We also prove in Sec.~4
that the dependence axis is preserved under convolution, implying that the covariance function of the linear RF $X$
also decays along this axis at
the smallest rate.

\smallskip

\noi {\it Notation.} In what follows, $C$ denote  generic positive constants
which may be different at different locations. We write $\limfdd$, $\eqfdd$, and $ \neqfdd $
for the weak convergence, equality, and inequality
of finite-dimensional
distributions, respectively.
$\boldsymbol{1} := (1,1)^\top$, $\boldsymbol{0} := (0,0)^\top$, $\R^2_0 :=  \R^2 \setminus \{ \boldsymbol{0}\}$, $\R_+^2 := \{  \boldsymbol{x} = (x_1,x_2)^\top \in \R^2: x_i > 0, \,
i=1,2 \}$, $\R_+ := (0,\infty)$ and $(\boldsymbol{0}, \boldsymbol{x}] := (0,x_1]\times (0,x_2]$, $\boldsymbol{x} = (x_1,x_2)^\top \in \R^2_+$. Also, $\lfloor x \rfloor := \max \{ k \in \Z: k \le x \}$,  $\lceil x \rceil := \min \{ k \in \Z: k \ge x \}$, $x \in \R$, and
$\lfloor {\mbx} \rfloor := (\lfloor x_1 \rfloor, \lfloor x_2 \rfloor)^\top$, $\lceil {\mbx} \rceil :=
(\lceil x_1 \rceil, \lceil x_2 \rceil)^\top$, $|\mbx| := |x_1|+|x_2|$, $\mbx = (x_1,x_2)^\top \in \R^2$.
We also write $f(\boldsymbol{x}) = f(x_1,x_2)$, $\boldsymbol{x} = (x_1,x_2)^\top \in \R^2$.





\section{Main results}

For $\gamma >0$, we study the limit distribution in \eqref{scale} of partial sums
\begin{eqnarray}\label{SX}
S_{\la,\gamma}({\mbx}) = \sum_{{\mbt} \in (0, \lambda x_1] \times (0, \lambda^\gamma x_2] \cap \Z^2} X({\mbt}), \quad {\mbx} \in \R^2_+,
\end{eqnarray}
over rectangles of a linear RF
\begin{equation}\label{Xlin2}
X({\mbt}) = \sum_{{\mbs} \in \Z^2} b({\mbt}-{\mbs}) \vep({\mbs}), \quad {\mbt} \in \Z^2,
\end{equation}
satisfying the following assumptions.

\medskip

\noi {\bf Assumption A} \ Innovations $\vep({\mbt})$, ${\mbt} \in \Z^2$, in \eqref{Xlin2}
are i.i.d.\ r.v.s with  $\E \vep (\boldsymbol{0}) =0, \E |\vep(\boldsymbol{0})|^2 =1$.

\medskip

\noi {\bf Assumption B} \ Coefficients $b({\mbt})$, ${\mbt} \in \Z^2$, in \eqref{Xlin2} satisfy
\begin{equation}\label{bcoef}
b(\boldsymbol{t}) = \rho(B \boldsymbol{t})^{-1} (L (B\boldsymbol{t}) + o(1)), \quad |\mbt| \to \infty,
\end{equation}
where $B = (b_{ij})_{i,j=1,2} $ is a real nondegenerate matrix, 
and
\begin{equation} \label{rh}
\rho (\mbu) := |u_1|^{q_1} + |u_2|^{q_2},  \quad \mbu \in \R^2,
\end{equation}
with $q_i>0$, $i=1,2$, satisfying \eqref{qLRD},
and
\begin{equation}\label{angL}
L(\boldsymbol{u}) := L_+ ( u_1/\rho(\boldsymbol{u})^{1/q_1} )\1 ( u_2 \ge 0 ) + L_-( u_1 / \rho(\boldsymbol{u})^{1/q_1} )\1 ( u_2 < 0 ), \quad \boldsymbol{u} \in \R^2_0,
\end{equation}
where $L_\pm (x)$, $x \in [-1,1]$, are continuous functions such that
$L_+ (1) = L_- (1)$, $L_+(-1)= L_- (-1). $ 

\medskip

We note that the boundedness and continuity assumptions of the `angular functions'  $L_\pm$ in \eqref{angL} do not seem necessary for our results and possibly can be relaxed. 
Note $q_1 < q_2$ for $1< Q <2 $  implies $H_1 \wedge H_2 = H_1,  \ \tilde H_1 \wedge \tilde H_2 = \tilde H_2.    $
Then from Proposition \ref{depaxis} we see that $\{ \mbt \in \R^2 : {\mbb}_2 \cdot {\mbt} = 0 \}$ with $\mbb_2 = (b_{21},b_{22})^\top$ is the dependence axis of $X$, which agrees with the coordinate axes 
if and only if $b_{21} = 0$ or $b_{22} = 0$ leading to the two cases, namely  $b_{21} = 0$  (congruous scaling) and
$b_{21} b_{22} \neq 0$ (incongruous scaling).  

The limit Gaussian RFs in our theorems are defined as stochastic integrals w.r.t.\ (real-valued) Gaussian white noise $W = \{W(\d \mbu), \, \mbu \in \R^2\}$
	with zero mean and variance $\E W(\d \mbu)^2 = \d \mbu $ (= the Lebesgue measure on $\R^2$).
Let
\begin{equation} \label{ainfty}
a_\infty({\mbu}) :=  \rho(\boldsymbol{u})^{-1} L (\boldsymbol{u}),  \quad {\mbu} \in \R^2_0,
\end{equation}
and
\begin{equation}\label{def:Vij}
\tilde V_{D} (\boldsymbol{x}) := |\operatorname{det}(B)|^{-\frac{1}{2}} \int_{\R^2} \Big\{ \int_{(\boldsymbol{0},\boldsymbol{x}]} a_\infty(\tilde D \boldsymbol{t}-\boldsymbol{u}) \d \boldsymbol{t}  \Big\} W(\d \boldsymbol{u}), \quad \boldsymbol{x} \in \R_+^2,
\end{equation}
where $\tilde D$ is any $2\times 2$ matrix  in \eqref{Bmat} below:
\begin{eqnarray} \label{Bmat}
&\tilde B_{00} := \begin{bmatrix}
b_{11}&0\\
0&b_{22}
\end{bmatrix}, \quad
\tilde B_{01} := \begin{bmatrix}
b_{11}&0\\
b_{21}&0
\end{bmatrix}, \quad
\tilde B_{02} := \begin{bmatrix}
0&b_{12}\\
0&b_{22}
\end{bmatrix}, \quad
\tilde B_{20} := \begin{bmatrix}
0&0\\
b_{21}&b_{22}
\end{bmatrix},\\
&\tilde B_{11} := \begin{bmatrix}
b_{11}&0\\
0&0
\end{bmatrix}, \quad
\tilde B_{21} := \begin{bmatrix}
0&0\\
b_{21}&0
\end{bmatrix}, \quad
\tilde B_{22} := \begin{bmatrix}
0&0\\
0&b_{22}
\end{bmatrix}.\nn
\end{eqnarray}
Also let
\begin{equation}\label{def:tildeVij}
V_{D}(\boldsymbol{x}) := |\operatorname{det}(B)|^{-1} \int_{\R^2} \Big\{ \int_{\R^2} a_\infty (\boldsymbol{t}) \1 (D \boldsymbol{t} + \boldsymbol{u} \in (\boldsymbol{0},\boldsymbol{x}] ) \d \boldsymbol{t}  \Big\} W (\d \boldsymbol{u}), \quad \boldsymbol{x} \in \R^2_+,
\end{equation}
where $D$ is any $2\times 2$ matrix  in \eqref{tildeBmat} below:
\begin{eqnarray} \label{tildeBmat}
B_{01} := \frac{1}{\operatorname{det}(B)} \left[
\begin{array}{rr}
b_{22}&0\\
-b_{21}&0
\end{array} \right], \quad
B_{10} := \frac{1}{\operatorname{det}(B)}
\begin{bmatrix}
b_{22}&-b_{12}\\
0&0
\end{bmatrix}, \quad B_{20} := \frac{1}{\operatorname{det}(B)}
\begin{bmatrix}
0&0\\
-b_{21}&b_{11}
\end{bmatrix},\\
B_{11} := \frac{1}{\operatorname{det}(B)} \left[
\begin{array}{rr}
b_{22}&0\\
0&0
\end{array} \right],\quad
B_{21} := \frac{1}{\operatorname{det}(B)} \left[
\begin{array}{rr}
0&0\\
-b_{21}&0
\end{array} \right], \quad
B_{22} := \frac{1}{\operatorname{det}(B)} \left[
\begin{array}{rr}
0&0\\
0&b_{11}
\end{array} \right]. \nn
\end{eqnarray}
Recall that in \eqref{Bmat}, \eqref{tildeBmat} $b_{ij}$ are entries of the matrix $B $ in \eqref{bcoef}.
To shorten notation, write $\tilde V_{ij} := \tilde V_{\tilde B_{ij}}$, $V_{ij} := V_{B_{ij}} $ and also define $\tilde V_0 :=  \tilde V_{B}$, $V_0 := V_{B^{-1}}$ satisfying $\tilde V_0 \eqfdd V_0$.
The existence of all these RFs in the corresponding regions of parameters $q_1, q_2$ is established
in Proposition \ref{Vexist}, which also identifies some of these RFs 
with FBS having one of its parameters
equal to $1$ or $\frac{1}{2}$. Recall that stochastic integral $\int_{\R^2} h({\mbu}) W(\d {\mbu}) =:  I(h)$ w.r.t.\
Gaussian white noise $ W$ is well-defined
for any $h \in L^2(\R^2)$ and has a Gaussian distribution with zero mean and variance
$\E |I(h)|^2 = \|h\|^2 = \int_{\R^2} |h({\mbu})|^2 \d {\mbu}$. Recall the definitions of $\tilde Q_i$, $\tilde H_i$, $H_i$, $i=1,2$, in
\eqref{Hnota}. In Proposition \ref{Vexist} and Theorems \ref{thmR11}--\ref{thmR1} below, Assumptions A and B
hold without further notice.  Let
$\tilde \sigma^2_{ij} := \E |\tilde V_{ij}({\bf 1})|^2$, $\sigma^2_{ij} := \E |V_{ij}({\bf 1})|^2$.

\begin{prop} \label{Vexist} 
	
	\medskip
	
	\noi (i) Let $q_1 < q_2 $ and $  \tilde Q_1 > 1 $. Then RFs $\tilde V_{00}$, $\tilde V_{11}$, $\tilde V_{20}$, $\tilde V_{21}$, $\tilde V_{22}$ are well-defined and
	\begin{equation} \label{R11}
	\tilde V_{22} \eqfdd \tilde \sigma_{22} B_{1, \tilde H_2}, \quad \tilde  V_{21} \eqfdd \tilde \sigma_{21} B_{\tilde H_2,1}, \quad
	\tilde V_{11} \eqfdd \tilde \sigma_{11} B_{\tilde H_1,1}.
	\end{equation}
	
	\medskip
	
	\noi (ii) Let $q_1 < q_2$ and $  \tilde Q_1 < 1 < \tilde Q_2 $. Then RFs $\tilde V_{00}$, $\tilde V_{11}$,
	$V_{01}$, $V_{11}$, $V_{21}$ are well-defined and
	\begin{equation} \label{R12}
	V_{11} \eqfdd \sigma_{11} B_{H_1, \frac{1}{2}}, \quad V_{21} \eqfdd \sigma_{21} B_{\frac{1}{2}, H_1}.
	\end{equation}
	
	\medskip
	
	\noi (iii) Let $q_1 < q_2 $ and $ \tilde Q_2 < 1 $. Then RFs $\tilde V_{00}$,
	$V_{01}$, $V_{11}$, $V_{21}$, $V_{22}$ are well-defined and
	\begin{equation} \label{R22}
	V_{22} \eqfdd \sigma_{22} B_{\frac{1}{2}, H_2}.
	\end{equation}
	
	\medskip
	
	\noi (iv) Let $q_1 = q_2 =: q \in (1, \frac{3}{2})$. Then $\tilde V_{02}$, $\tilde V_0 $ and $\tilde V_{01} $ are well defined and
	\begin{equation} \label{R02}
	\tilde V_{01} \eqfdd \tilde \sigma_{01} B_{\tilde H,1}, \quad \tilde V_{02} \eqfdd \tilde \sigma_{02} B_{1, \tilde H}, \quad \tilde H := 2-q \in \big( \frac{1}{2}, 1 \big).
	\end{equation}
	
	\medskip
	
	\noi (v) Let $q_1 = q_2 =: q \in (\frac{3}{2}, 2) $. Then $V_{10}$, $V_0 $ and $V_{20} $ are well defined and
	\begin{equation} \label{R10}
	V_{10} \eqfdd \sigma_{10} B_{H, \frac{1}{2}}, \quad
	V_{20} \eqfdd \sigma_{20} B_{\frac{1}{2},H}, \quad H := \frac{5}{2}-q \in \big( \frac{1}{2},1 \big).
	\end{equation}
\end{prop}

\noi As noted above, our main results  (Theorems \ref{thmR11}--\ref{thmR1})
describe the anisotropic scaling limits and the scaling transition of the linear RF $X$ in \eqref{Xlin2}, viz.,
\begin{equation}\label{scale1}
\big\{A^{-1}_{\la, \gamma} S^X_{\la, {\gamma}}({\mbx}), \, \mbx \in \R^2_+ \big\}  \limfdd V^X_{\gamma} =
\begin{cases} V^X_+,  &\gamma > \gamma^X_0, \\
V^X_-, &\gamma < \gamma^X_0, \\
V^X_0, &\gamma = \gamma^X_0,
\end{cases}
\quad \lambda \to \infty,
\end{equation}
where $S^X_{\la, {\gamma}}$ is the partial-sum RF in \eqref{SX}.


\begin{thm} \label{thmR11} Let $q_1 < q_2 $ and  $\tilde Q_{1} > 1 $. Then the convergence in \eqref{scale1} holds for all $\gamma >0$ in
	
	\medskip
	
	\noi (i) Case $b_{21} b_{22} \ne 0$ (incongruous scaling) with
	$\ga^X_0 = 1$, $V^X_+ = \tilde V_{22}$, $V^X_- = \tilde V_{21}$, $V^X_0 = \tilde V_{20}$.
	
	\medskip
	
	\noi  (ii) Case $b_{21} = 0 $ (congruous scaling) with
	$ \ga^X_0 = \frac{q_1}{q_2}$, $V^X_+ =  \tilde V_{22}$, $V^X_- = \tilde V_{11}$, $V^X_0 = \tilde V_{00}$.
\end{thm}

\begin{thm} \label{thmR12} Let $q_1 < q_2 $ and $\tilde Q_{1} < 1  <  \tilde Q_{2}$. Then the convergence in \eqref{scale1}
	holds for all $\gamma >0 $ in
	
	\medskip
	
	\noi (i) Case $b_{21} b_{22} \ne 0 $ (incongruous scaling) with
	$\ga^X_0 = 1$, $V^X_+ = V_{11}$, $V^X_- =  V_{21}$, $V^X_0 = V_{01}$.
	
	\medskip
	
	\noi (ii) Case $b_{21} = 0 $ (congruous scaling) with
	$\ga^X_0 = \frac{q_1}{q_2}$, $V^X_+ = V_{11}$, $V^X_- = \tilde V_{11}$, $V^X_0 = \tilde V_{00}$.
	
\end{thm}

\begin{thm} \label{thmR22} Let  $q_1 < q_2 $ and $\tilde Q_{2}< 1$. Then the convergence in \eqref{scale1}
	holds for all $\gamma >0 $ in
	
	\medskip
	
	\noi (i) Case $b_{21} b_{22} \ne 0 $ (incongruous scaling) with $\ga^X_0$, $V^X_+$, $V^X_-$, $V^X_0 $ the same as in Theorem \ref{thmR12} (i).
	
	\medskip
	
	\noi (ii) Case $b_{21} = 0 $ (congruous scaling) with
	$\ga^X_0 = \frac{q_1}{q_2}$, $V^X_+ =  V_{11}$, $V^X_- =  V_{22}$, $V^X_0 = \tilde V_{00}$.

\end{thm}

Theorem \ref{thmR1}  discusses the case $q_1 = q_2$ when the dependence axis is undefined.

\begin{thm} \label{thmR1}  Let $q_1 = q_2=:q$ and $\tilde Q_1 = \tilde Q_2 =: \tilde Q$.  Then the convergence in \eqref{scale1}
	holds for all $\gamma >0 $ in
	
	\medskip
	
	\noi (i) Case $q\in (1,\frac{3}{2})$ or $\tilde Q > 1 $ with $\ga^X_0 = 1$, $V^X_+ = \tilde V_{02}$, $V^X_- = \tilde V_{01}$, $V^X_0 = \tilde V_{0}$.
	
	\medskip
	
	\noi (ii) Case $q\in (\frac{3}{2},2)$ or $\tilde Q < 1 $ with $\ga^X_0 = 1$, $V^X_+ = V_{10}$, $V^X_- = V_{20}$, $V^X_0 = V_{0}. $
	
\end{thm}

\begin{rem} {\rm  In the above theorems the convergence in \eqref{scale1} holds under normalization
		\begin{eqnarray}\label{Ala}
		A_{\la, \ga} = \la^{H(\ga)},
		\end{eqnarray}
		where $H(\ga) >0$ is defined in the proof of these theorems below. Under congruous scaling  $b_{21} = 0$  the exponent $H(\ga)$ in
		\eqref{Ala} is the same
		as in the case $B = I $ (= the identity matrix) studied in \cite{pils2017}.
		As shown in \cite{ps2016}, $V^X_\ga $ in \eqref{scale1} satisfies the following self-similarity property:
		\begin{equation}\label{ssH}
		\big\{ V^X_\ga (\la^\Gamma{\mbx}), \, \mbx \in \R^2_+ \big\} \eqfdd \big\{ \la^{H(\ga)} V^X_\ga ({\mbx}), \, \mbx \in \R^2_+ \big\} \quad 
		\forall \la >0,
		\end{equation}
		where 
		$\la^\Gamma = \operatorname{diag}(\la, \la^\ga)$ and $H(\ga)$ is the same as in \eqref{Ala}.
		Note that an FBS  $B_{{\cal H}_1, {\cal H}_2}$ with ${\cal H}_i \in (0,1]$, $i=1,2$, satisfies \eqref{ssH} with
		\begin{equation}\label{Hga}
		H(\ga) = {\cal H}_1 + \ga {\cal H}_2.
		\end{equation}
		Thus, in Theorems \ref{thmR11}--\ref{thmR1} in the case of unbalanced (FBS) limits $A_{\la,\gamma}$ in \eqref{Ala}
		can also be identified from \eqref{Hga} and the expressions for $H_i$, $\tilde H_i$, $i=1,2$, in \eqref{Hnota}.}
	
\end{rem}

\section{Proofs of Proposition \ref{Vexist} and Theorems \ref{thmR11}--\ref{thmR1} }


For $\gamma >0$, the limit distribution of $S^X_{\lambda, \gamma}$ is obtained using a general criterion
for the weak convergence of linear forms in i.i.d.\ r.v.s towards a stochastic integral w.r.t.\ the white noise.
Consider a linear form
\begin{equation} \label{Sg}
S(g) := \sum_{{\mbs} \in \Z^2}  g({\mbs}) \vep({\mbs})
\end{equation}
with real coefficients $\sum_{{\mbs} \in \Z^2} g({\mbs})^2 < \infty$ and innovations satisfying Assumption A.
The following proposition
extends (\cite{book2012}, Prop.~14.3.2), (\cite{sur2019b}, Prop.~5.1), (\cite{sur2019a}, Prop.~3.1).

\begin{prop}\label{prop}
	For $\lambda >0$, let $S(g_\lambda)$ be as in \eqref{Sg}.
	Assume that for some $2 \times 2$ non-degenerate matrix $A$ and $\Lambda = \operatorname{diag}(l_1, l_2)$ with $l_i = l_i(\lambda)$, $i=1,2$, such that $l_1 \wedge l_2 \to \infty$, $\lambda \to\infty$, the functions
	\begin{equation}\label{def:tildegn}
	\tilde g_\lambda (\boldsymbol{\boldsymbol{u}}) := |\operatorname{det}(A\Lambda)|^{1/2} g_\lambda (\lceil A \Lambda \boldsymbol{u} \rceil), \quad \boldsymbol{u} \in \R^2, \quad \lambda > 0,
	\end{equation}
	tend to a limit $h$ in $L^2(\R^2)$, i.e.\
	\begin{equation}\label{L2conv}
	\| \tilde g_\lambda - h \|^2 = \int_{\R^2} | \tilde g_\lambda (\boldsymbol{u}) - h(\boldsymbol{u}) |^2 \d \boldsymbol{u} \to 0, \quad \lambda \to \infty.
	\end{equation}
	Then
	\begin{equation}\label{lim:d}
	S(g_\lambda) \limd I (h) = \int_{\R^2} h(\boldsymbol{u}) W(\d \boldsymbol{u}), \quad \lambda \to \infty.
	\end{equation}
\end{prop}

\begin{proof}
Denote by $S(\R^2)$ the set of simple functions $f : \R^2 \to \R$, which are finite linear combinations of indicator functions of disjoint squares $\Box_{\boldsymbol{k}}^K := \prod_{i=1}^2 (k_i/K, (k_i+1)/K]$,
$\boldsymbol{k} \in \Z^2$, $K \in \N$. The set $S(\R^2)$ is dense in $L^2(\R^2)$: given $f \in L^2(\R^2)$, for every $\epsilon >0$ there exists $f_\epsilon \in S(\R^2)$ such that $\| f-f_\epsilon\|<\epsilon$. \eqref{lim:d} follows once we show that for every $\epsilon>0$ there exists $h_\epsilon \in S(\R^2)$ such that as $\lambda \to \infty$, the following relations (i)--(iii) hold:
(i) $\E |S(g_\lambda)-S(h_{\epsilon,\lambda})|^2 < \epsilon$,
(ii) $S(h_{\epsilon,\lambda}) \limd I (h_\epsilon)$,
(iii) $\E|I (h_\epsilon)-I (h)|^2 < \epsilon$, where
\begin{equation}\label{def:hepsla}
h_{\epsilon,\lambda} (\boldsymbol{s}) := |\operatorname{det}(A\Lambda)|^{-1/2} h_\epsilon( (A\Lambda)^{-1} \boldsymbol{s} ), \quad \boldsymbol{s} \in \Z^2, \quad \lambda>0.
\end{equation}
As for (i), note that
\begin{align*}
\E |S(g_\lambda)-S(h_{\epsilon,\lambda})|^2 &= \int_{\R^2} |g_\lambda(\lceil \boldsymbol{s} \rceil)-h_{\epsilon,\lambda}(\lceil \boldsymbol{s} \rceil)|^2 \d \boldsymbol{s}\\
&=  |\operatorname{det}(A\Lambda)| \int_{\R^2} | g_\lambda ( \lceil A \Lambda \boldsymbol{u} \rceil ) - h_{\epsilon,\lambda} (\lceil A \Lambda \boldsymbol{u} \rceil) |^2 \d \boldsymbol{u}
= \| \tilde g_\lambda - \tilde h_{\epsilon,\lambda} \|^2,
\end{align*}
where $\tilde h_{\epsilon,\lambda}$ is derived from $h_{\epsilon,\lambda}$ in the same way as $\tilde g_\lambda$ is derived from $g_\lambda$ in \eqref{def:tildegn}. To prove (i) we need to find suitable  $h_\epsilon \in S(\R^2)$ and thus $h_{\epsilon,\lambda}$ in \eqref{def:hepsla}. By \eqref{L2conv}, there exists $\lambda_0 >0$ such that $\| \tilde g_\lambda - h \| < \epsilon/4$,
$\forall \lambda \ge \lambda_0$. Given $\tilde g_{\lambda_0} \in L^2(\R^2)$, there exists  $h_\epsilon \in S(\R^2)$ such that $\| \tilde g_{\lambda_0} - h_\epsilon \| < \epsilon/4$. Note that
$$
\| h_{\epsilon} - \tilde h_{\epsilon,\lambda} \|^2 = \int_{\R^2} |h_\epsilon(\boldsymbol{u}) - h_\epsilon ((A\Lambda)^{-1} \lceil A \Lambda \boldsymbol{u} \rceil ) |^2 \d \boldsymbol{u} \to 0, \quad \lambda \to \infty,
$$
follows from $|(A\Lambda)^{-1} \lceil A \Lambda \boldsymbol{u} \rceil - \boldsymbol{u}|
= |(A \Lambda)^{-1}(\lceil A \Lambda \boldsymbol{u} \rceil - A \Lambda \boldsymbol{u})| \le C \min(l_1, l_2)^{-1} = o(1)$ uniformly
in $\boldsymbol{u} \in \R^2$
and the fact that $h_\epsilon$ is bounded and has a compact support. Thus, there exists $\lambda_1 >0$ such that  $\| h_\epsilon - \tilde h_{\epsilon,\lambda} \| < \epsilon/4$, $\forall \lambda \ge \lambda_1 $.
Hence, 
$$\| \tilde g_\lambda - \tilde h_{\epsilon,\lambda} \| \le \| \tilde g_\lambda - h \| + \| h - \tilde g_{\lambda_0} \| + \| \tilde g_{\lambda_0} - h_\epsilon \| + \| h_\epsilon - \tilde h_{\epsilon,\lambda} \| < \epsilon, \quad \forall \lambda \ge \lambda_0 \vee \lambda_1,
$$
completing the proof of (i). The above reasoning implies also  (iii) since
$(\E | I(h_\epsilon) - I(h) |^2)^{1/2} = \| h_\epsilon - h \| \le \| h_\epsilon - \tilde g_{\lambda_0} \| + \| \tilde g_{\lambda_0} - h \| < \epsilon/2$.

It remains to prove (ii). The step function $h_\epsilon$ in the above proof of (i) can be written as
$h_\epsilon (\boldsymbol{u}) = \sum_{\boldsymbol{k} \in \Z^2} h_\epsilon^{\Box_{\boldsymbol{k}}^K} \1 ( \boldsymbol{u} \in \Box_{\boldsymbol{k}}^K )$,  $\boldsymbol{u} \in \R^2$,
$(\exists K \in \N)$, where $h_\epsilon^{\Box_{\boldsymbol{k}}^K} = 0$ except for a finite number of 
$\boldsymbol{k} \in \Z^2$.
Then, by \eqref{def:hepsla},
$$
h_{\epsilon,\lambda} (\boldsymbol{s}) = |\operatorname{det}(A\Lambda)|^{-1/2} \sum_{\boldsymbol{k} \in \Z^2} h_\epsilon^{\Box_{\boldsymbol{k}}^K} \1 ( \boldsymbol{s} \in A \Lambda \Box_{\boldsymbol{k}}^K ), \quad \boldsymbol{s} \in \Z^2,
$$
and $
S(h_{\epsilon,\lambda}) =  \sum_{\boldsymbol{k} \in \Z^2} h_\epsilon^{\Box^K_{\boldsymbol{k}}} W_\lambda (\Box_{\boldsymbol{k}}^K), $
where
$$
W_\lambda (\Box_{\boldsymbol{k}}^K) :=  |\operatorname{det}(A\Lambda)|^{-1/2} \sum_{\boldsymbol{s} \in A \Lambda \Box_{\boldsymbol{k}}^K} \varepsilon(\boldsymbol{s}), \quad \boldsymbol{k} \in \Z^2.
$$
Since the r.v.s $\vep(\boldsymbol{s})$, $\boldsymbol{s} \in \Z^2$, are i.i.d.\ with $\E \varepsilon (\boldsymbol{0}) = 0$, $\E |\varepsilon (\boldsymbol{0})|^2 = 1$ and the parallelograms $A \Lambda \square_{\boldsymbol{k}}^K$, $\boldsymbol{k} \in \Z^2$, are disjoint, the r.v.s $W_\lambda (\Box_{\boldsymbol{k}}^K)$, $\boldsymbol{k} \in \Z^2$, are independent and satisfy $\E W_\lambda (\Box_{\boldsymbol{k}}^K) = 0$,
$\E | W_\lambda (\Box_{\boldsymbol{k}}^K ) |^2 = |\operatorname{det}(A\Lambda)|^{-1} \sum_{\boldsymbol{s} \in \Z^2} \1 ( \boldsymbol{s} \in A \Lambda \Box_{\boldsymbol{k}}^K) \to \int_{\Box_{\boldsymbol{k}}^K} \d \boldsymbol{u},$ $\lambda \to \infty$.
Hence, by the classical CLT, for every $J \in \N$,
$$
\big\{ W_\lambda (\Box^K_{\boldsymbol{k}}), \, \boldsymbol{k} \in \{-J, \dots, J\}^2 \big\}
\limd \big\{W(\Box_{\boldsymbol{k}}^K), \, \boldsymbol{k} \in \{-J, \dots, J\}^2 \big\}, \quad \lambda \to \infty,
$$
implying the convergence
$S(h_{\epsilon,\lambda}) \limd \sum_{\boldsymbol{k} \in \Z^2} h_\epsilon^{\Box_{\boldsymbol{k}}^K} W(\Box_{\boldsymbol{k}}^K ) = I(h_\epsilon)$, $\lambda \to \infty$, or part (ii), and
completing the proof of the proposition.
\end{proof}

We shall also need some properties of the generalized homogeneous function $\rho$ in
 \eqref{rh} for $q_i>0$, $i=1,2$, with $Q:=\frac{1}{q_1}+\frac{1}{q_2}$. Note 
 the elementary inequality
\begin{equation}\label{rhoineq}
C_1 \rho(\mbu)^{1/q_1}  \le (|u_1|^2 + |u_2|^{2q_2/q_1})^{1/2} \le C_2 \rho(\mbu)^{1/q_1}, \quad \mbu \in \R^2,
\end{equation}
with $C_i >0$, $i=1,2$, independent of $\mbu $, see   (\cite{sur2019a}, (2.16)). From \eqref{rhoineq} and  (\cite{pils2017}, Prop.~5.1)
we obtain for any $\delta >0$,
\begin{equation}\label{intrho}
\int_{\R^2} \rho (\boldsymbol{u})^{-1} \1 ( \rho(\boldsymbol{u}) < \delta ) \d \boldsymbol{u}
< \infty \Longleftrightarrow Q>1, \quad \int_{\R^2} \rho (\boldsymbol{u})^{-2} \1 ( \rho(\boldsymbol{u}) \ge \delta ) \d \boldsymbol{u}
< \infty \Longleftrightarrow Q<2.
\end{equation}
Moreover,  with $q= \max\{ q_1, q_2, 1\}$,
\begin{equation} \label{rhoineq2}
\rho(\boldsymbol{u}+\boldsymbol{v})^{1/q}
	\le \rho(\boldsymbol{u})^{1/q} + \rho(\boldsymbol{v})^{1/q}, \quad   \boldsymbol{u}, \boldsymbol{v} \in \R^2,
\end{equation}	
see  (\cite{pils2017}, (7.1)), and, for $1< Q < 2$,
\begin{equation}\label{def:conv}
(\rho^{-1} \star \rho^{-1}) (\boldsymbol{u}) := \int_{\R^2} \rho (\boldsymbol{v})^{-1} \rho (\boldsymbol{v} + \boldsymbol{u})^{-1} \d \boldsymbol{v}
= \tilde \rho (\boldsymbol{u})^{-1} \tilde  L\big( |u_1|/\tilde \rho (\boldsymbol{u})^{1/\tilde q_1} \big), \quad \boldsymbol{u} \in \R^2_0,
\end{equation}
where with $\tilde q_i := q_i (2-Q)$, $i=1,2$,
\begin{equation}\label{tilderho}
\tilde \rho (\boldsymbol{u}) := |u_1|^{\tilde q_1}+|u_2|^{\tilde q_2}, \quad \boldsymbol{u} \in \R^2,
\end{equation}
and $\tilde L (z) := (\rho^{-1} \star \rho^{-1}) (z,(1-z^{\tilde q_1})^{1/\tilde q_2})$,
$z \in [0,1]$, is  a continuous function. Note $\tilde q_2 < 1 $ (respectively, $\tilde q_1 < 1 $) is
equivalent to $\tilde Q_1 > 1 $ (respectively, $\tilde Q_2 > 1$). The proof of \eqref{def:conv}
is similar to that of (\cite{pils2017}, (5.6))   (see also Proposition \ref{conaxis} below).

\begin{proof}[Proof of Proposition \ref{Vexist}]
Since $\tilde V_{ij}(\boldsymbol{1}) = I (\tilde h_{ij})$, $V_{ij}(\boldsymbol{1}) = I (h_{ij})$,
it suffices to prove
\begin{equation} \label{hij}
\|\tilde h_{ij}\| < \infty, \quad \| h_{ij}\| < \infty
\end{equation}
for suitable $i,j$ in the corresponding regions of parameters $q_1, q_2$.
Using the boundedness of $L_\pm $ in \eqref{ainfty} and \eqref{rhoineq}
we can replace $|a_\infty| $ by $\rho^{-1} $
in the subsequent proofs of \eqref{hij}. Hence and from \eqref{def:conv} it follows that
\begin{align}\label{hconv}
\| \tilde h_{ij} \|^2 \le C \int_{(0,1]^2 \times (0,1]^2} (\rho^{-1} \star \rho^{-1} ) ( \tilde B_{ij} (\boldsymbol{t}-\boldsymbol{s}) ) \d \boldsymbol{t} \d \boldsymbol{s} \le C \int_{(0,1]^2 \times (0,1]^2}  \tilde \rho ( \tilde B_{ij}(\boldsymbol{t}-\boldsymbol{s}))^{-1} \d \boldsymbol{t} \d \boldsymbol{s}.
\end{align}

\smallskip

\noi \underline{Existence of $\tilde V_{00}$.} Relation $\|\tilde h_{00}\|^2 \le C \int_{[-1,1]^2} \tilde \rho(b_{11}t_1,b_{22} t_2)^{-1} \d t_1 \d t_2 < \infty$ follows from \eqref{intrho} since $\frac{1}{\tilde q_1} + \frac{1}{\tilde q_2} = \frac{Q}{2-Q}>1$ for $1<Q<2$.
This proves the existence of $\tilde V_{00}$ in all cases (i)--(iii) of Proposition \ref{Vexist}.

\smallskip

\noi \underline{Existence of $\tilde V_{20}$, $\tilde V_{21} $, $\tilde V_{22}$.}  From \eqref{hconv} we get $\|\tilde h_{20}\|^2 \le
C \int_{ [-1,1]^2} |b_{21} t_1 + b_{22} t_2|^{-\tilde q_2} \d t_1 \d t_2    < \infty $ since
$\tilde q_2 < 1 $. The proof of $\|\tilde h_{21}\| < \infty$, $\|\tilde h_{22}\| < \infty $ is completely analogous.
This proves the existence of $\tilde V_{20}$, $\tilde V_{21} $ and $\tilde V_{22}$  for $\tilde Q_1 > 1 $.

\smallskip

\noi \underline{Existence of  $\tilde V_{11}$.} Similarly as above, from \eqref{hconv} we get
$\|\tilde h_{11}\|^2 \le  
C \int_{ [-1,1]} |b_{11} t|^{-\tilde q_1} \d t< \infty$ since $\tilde q_1>1$. This proves the existence of $\tilde V_{11}$ for $\tilde Q_2 > 1 $.

\smallskip

\noi \underline{Existence of  $V_{01}$, $V_{11}$, $V_{21}$.}  We have
\begin{align*}
\|h_{01}\|^2
&\le C\int_{\R^2} \Big( \int_{\R} |t_1|^{-q_1(1-\frac{1}{q_2})} \1  \big( b_{22} t_1 + u_1 \in (0,1], -b_{21} t_1 + u_2 \in (0,1] \big) \d t_1 \Big)^2 \d \boldsymbol{u} \\
&\le C\int_{\R^2} |t_1|^{-q_1(1-\frac{1}{q_2})}  |t_2|^{-q_1(1-\frac{1}{q_2})} \1(|b_{22}(t_1-t_2)| \le 1, |b_{21}(t_1-t_2)| \le 1)
 \d t_1 \d t_2 \\
&\le C\int_{\R^2} |t_1|^{-q_1(1-\frac{1}{q_2})}  |t_2|^{-q_1(1-\frac{1}{q_2})} \1(|t_1-t_2| \le 1)
 \d t_1 \d t_2  < \infty
\end{align*}
since $\frac{1}{2} < q_1(1-\frac{1}{q_2}) < 1 $ or $\tilde Q_1 < 1 < Q$. The proof of $\|h_{11}\| < \infty $
and  $\|h_{21}\| < \infty $ is completely analogous.

\smallskip

\noi \underline{Existence of  $V_{22}$.} Similarly as above,
$\|h_{22}\|^2
\le C\int_{\R^2} |t_1|^{-q_2(1-\frac{1}{q_1})}  |t_2|^{-q_2(1-\frac{1}{q_1})} \1(|t_1-t_2| \le 1)
 \d t_1 \d t_2  < \infty $ since $\frac{1}{2} < q_2(1-\frac{1}{q_1}) < 1 $ or $\tilde Q_2 < 1 < Q$.

\smallskip

\noi \underline{Existence of  $\tilde V_{02}$, $\tilde V_{01}$.} From \eqref{hconv} we get
$\| \tilde h_{02} \|^2 \le C \int_{[-1,1]^2}  \tilde \rho (b_{12}t_2,b_{22}t_2)^{-1} \d \boldsymbol{t} < \infty $ since
$\tilde q = 2(q-1) < 1 $ for $q \in (1,\frac{3}{2})$. The proof of  $\|\tilde  h_{01} \| < \infty $ is completely analogous.

\smallskip

\noi \underline{Existence of  $V_{20}$, $V_{10}$.}
We have
\begin{align*}
\| h_{20} \|^2 &\le C \int_{\R} \Big( \int_{\R^2} \rho(\mbt)^{-1} \1 (b_{22} t_1-b_{12}t_2+u \in (0,1]) \d \mbt \Big)^2 \d u\\
&\le C \int_{\R} \Big( \int_{\R} \rho \big( 1+\frac{b_{12}}{b_{22}}t_2,t_2 \big)^{-1} \d t_2 \int_{\R} |t_1|^{1-q} \1 (b_{22} t_1+u \in (0,1]) \d t_1 \Big)^2 \d u< \infty
\end{align*}
for $q \in (\frac{3}{2},2)$. The proof of  $\|h_{10} \| < \infty $ is completely analogous.
%

\smallskip

\noi \underline{Existence of  $\tilde V_0$, $V_0$.} Relation $\| \tilde h_0 \|^2 \le C \int_{[-1,1]^2} \tilde \rho(B\mbt)^{-1} \d \mbt < \infty$ follows from \eqref{intrho} since $\frac{2}{\tilde q}= \frac{1}{q-1}>1$ for $q \in (1,2)$. We have $\| h_0 \| = \| \tilde h_0 \| < \infty$.

\smallskip

It remains to show the relations \eqref{R11}--\eqref{R10}, which follow from the variance expressions: for any ${\mbx} \in \R^2_+ $,
we have that
\begin{eqnarray}\label{VB}
&\E |\tilde V_{22}({\mbx})|^2 = \tilde \sigma^2_{22} x_1^2 x_2^{2\tilde H_2}, \quad \E |\tilde V_{21}({\mbx})|^2 = \tilde \sigma^2_{21}
x_1^{2\tilde H_2} x_2^2, \quad
\E |\tilde V_{11}({\mbx})|^2 = \tilde \sigma^2_{11} x_1^{2\tilde H_1} x_2^2,  \\
&\E |V_{11}({\mbx})|^2 = \sigma^2_{11} x_1^{2H_1} x_2, \quad \E |V_{21}({\mbx})|^2 =
\sigma^2_{21} x_1 x_2^{2H_1}, \quad
\E |V_{22}({\mbx})|^2 = \sigma^2_{22} x_1 x_2^{2H_2}, \nn \\
&\E |\tilde V_{01}({\mbx})|^2 = \tilde \sigma^2_{01} x_1^{2\tilde H} x_2^2, \quad \E |\tilde V_{02}({\mbx})|^2 = \tilde \sigma^2_{02} x_1^2 x_2^{2\tilde H}, \quad
\E |V_{10}({\mbx})|^2 = \sigma^2_{10} x_1^{2H} x_2, \quad
\E |V_{20}({\mbx})|^2 = \sigma^2_{20} x_1 x_2^{2H}. \nn
\end{eqnarray}
Relations \eqref{VB} follow by a change of variables in the corresponding integrals, using the invariance
property: $\la a_\infty(\la^{\frac{1}{q_1}} t_1, \la^{\frac{1}{q_2}} t_2) = a_\infty (\boldsymbol{t})$, $\boldsymbol{t} \in \R^2_0$, for all $\la  >0$.  E.g., after a change of variables $t_2 \to x_2 t_2$, $u_2 \to x_2 u_2$, $u_1 \to x_2^{\frac{q_2}{q_1}} u_1 $,
the first expectation in \eqref{VB} writes as
$ \E |\tilde V_{22}({\mbx})|^2 = x_1^2 |\operatorname{det} (B)|^{-1}
\int_{\R^2} |\int_0^{x_2} a_\infty (u_1, \allowbreak b_{22}t_2+u_2) \d t_2 |^2 \d {\mbu}
= x_1^2 x_2^{2\tilde H_2} \tilde \sigma^2_{22}$, where
$ \tilde \sigma^2_{22} :=  |\operatorname{det} (B)|^{-1} \int_{\R^2} |\int_0^{1} a_\infty (u_1, b_{22}t_2+u_2) \d t_2 |^2 \d {\mbu} < \infty$.  Proposition~\ref{Vexist} is proved.
\end{proof}

\medskip

Similarly as in \cite{pils2017,sur2019a} and other papers, in Theorems \ref{thmR11}--\ref{thmR22}
we restrict the proof of \eqref{scale1} to one-dimensional convergence at 
$\mbx \in \R^2_+$.
Towards this end, we use Proposition \ref{prop} and rewrite every
$\lambda^{-H(\gamma)}  S^X_{\la,\gamma}(\mbx) \allowbreak = S(g_\la)$ as a linear form in \eqref{Sg} with
\begin{equation} \label{gla1}
g_\la ({\mbu}) := \la^{-H(\gamma)}  \int_{(0, \la x_1] \times (0, \la^\gamma x_2]}
b(\lceil {\mbt} \rceil - {\mbu})\d {\mbt}, \quad \mbu \in \Z^2.
\end{equation}
In what follows, w.l.g., we set $|\operatorname{det}(B)| =1$.

\begin{proof}[Proof of Theorem \ref{thmR11}] \underline{Case (i) and  $\gamma > 1$ or $V^X_+ = \tilde V_{22}$.} Set
$H(\gamma ) =  1 + \gamma (\frac{3}{2}+ \frac{q_2}{2q_1} - q_2) =
1 + \gamma \tilde H_2 $ in agreement with \eqref{Ala}--\eqref{Hga}.
Rewrite
\begin{equation*}
g_{\lambda}(\boldsymbol{s})
= \lambda^{1+\gamma-H(\gamma)} \int_{\R^2} b(\lceil \Lambda' \boldsymbol{t} \rceil - \boldsymbol{s}) \1 ( \Lambda' \boldsymbol{t} \in (\boldsymbol{0}, \lfloor \Lambda' \boldsymbol{x} \rfloor] ) \d \boldsymbol{t}, \quad \boldsymbol{s} \in \Z^2,
\end{equation*}
with $\Lambda' = \operatorname{diag}(\lambda,\lambda^\gamma)$.
Use Proposition \ref{prop} with $A = B^{-1}$ and $\Lambda = \operatorname{diag}(l_1,l_2)$, where $l_1 = \lambda^{\gamma \frac{q_2}{q_1}}$,
$l_2 = \lambda^{\gamma}$. According to the definition in \eqref{def:tildegn},
\begin{align}
\tilde g_\lambda(\boldsymbol{u}) = 
\int_{\R^2}  \lambda^{\gamma q_2}
b ( \lceil \Lambda' \boldsymbol{t} \rceil - \lceil B^{-1}  \Lambda \boldsymbol{u} \rceil ) \1 (
\Lambda' \boldsymbol{t} \in
(\boldsymbol{0}, \lfloor \Lambda' \boldsymbol{x} \rfloor]
) \d \boldsymbol{t}, \quad \boldsymbol{u} \in \R^2, \label{tildeg}
\end{align}
for which we need to show the $L^2$-convergence in \eqref{L2conv} with $h (\boldsymbol{u})$ replaced by
\begin{equation}\label{h22}
\tilde h_{22} (\boldsymbol{u}) := 
\int_{(\boldsymbol{0},\boldsymbol{x}]} a_\infty (\tilde B_{22} \boldsymbol{t} - \boldsymbol{u}) \d \boldsymbol{t}, \quad \boldsymbol{u} \in \R^2,
\end{equation}
where the integrand does not depend on $t_1$.
Note since $q_1 < q_2 $ and $\gamma>1$ that
\begin{equation}\label{point}
\Lambda^{-1} B (\lceil \Lambda' \boldsymbol{t} \rceil - \lceil B^{-1} \Lambda \boldsymbol{u} \rceil) \to \tilde B_{22} \boldsymbol{t} - \boldsymbol{u},
\end{equation}
point-wise for any ${\mbt}, {\mbu} \in \R^2 $ and
therefore, by continuity of $\rho^{-1}$,
\begin{equation}\label{blim1}
\lambda^{\gamma q_2} \rho \big(B (\lceil \Lambda' \boldsymbol{t} \rceil - \lceil B^{-1} \Lambda \boldsymbol{u} \rceil ) \big)^{-1}
\to \rho(\tilde B_{22}\boldsymbol{t} - \boldsymbol{u})^{-1}
\end{equation}
for any $\boldsymbol{t},\boldsymbol{u} \in \R^2$ such that  $\tilde B_{22} \boldsymbol{t}-\boldsymbol{u} \neq \boldsymbol{0}$.
Later use \eqref{bcoef}, \eqref{point}, \eqref{blim1} and continuity of $L_\pm$ to get
\begin{equation}\label{lab1}
\lambda^{\gamma q_2}
b ( \lceil \Lambda' \boldsymbol{t} \rceil - \lceil B^{-1} \Lambda \boldsymbol{u} \rceil ) \to
a_\infty( \tilde B_{22} \boldsymbol{t} - \boldsymbol{u})
\end{equation}
for $\boldsymbol{t}, \boldsymbol{u} \in \R^2$ such that  $\tilde B_{22} \boldsymbol{t}-\boldsymbol{u} \neq \boldsymbol{0}$.
Therefore, $\tilde g_\lambda(\mbu) \to \tilde h_{22}(\mbu)$ 
for all $\mbu \in \R^2$.
This point-wise convergence can be extended to that in $L^2(\R^2)$ by applying Pratt's lemma,
c.f.\ (\cite{pils2017}, proof of Theorem 3.2), to the domination $|\tilde g_\lambda(\boldsymbol{u})| \le C \tilde G_\lambda(\boldsymbol{u}) $, where
\begin{align} \label{tildeG}
\tilde G_\lambda(\boldsymbol{u}) &:= \int_{(\boldsymbol{0},\boldsymbol{x}]} \rho(\Lambda^{-1}B\Lambda' \boldsymbol{t} - \boldsymbol{u})^{-1} \d \boldsymbol{t}
\to \int_{(\boldsymbol{0},\boldsymbol{x}]} \rho(\tilde B_{22} \boldsymbol{t}-\boldsymbol{u})^{-1} \d \boldsymbol{t}
=: \tilde G(\boldsymbol{u}),
\end{align}
for all $\mbu \in \R^2$. To get this domination we use
$| b(\boldsymbol{s}) | \le C \max\{ \rho(B\boldsymbol{s}), 1 \}^{-1}$, $\boldsymbol{s} \in \Z^2$, and by \eqref{rhoineq2} and $\rho (\Lambda^{-1} \mbt) = \lambda^{-\gamma q_2} \rho(\mbt)$, we further see that
\begin{align*}
&\rho ( \Lambda^{-1} B \Lambda' \boldsymbol{t} - \boldsymbol{u} )\nn \\
&\quad \le C\big\{ \rho \big( \Lambda^{-1} B (\lceil \Lambda' \boldsymbol{t} \rceil -
\lceil B^{-1} \Lambda \boldsymbol{u} \rceil ) \big) + \rho \big(\Lambda^{-1} B ( \Lambda' \boldsymbol{t} -
B^{-1} \Lambda \boldsymbol{u} - \lceil \Lambda' \boldsymbol{t} \rceil + \lceil B^{-1} \Lambda \boldsymbol{u} \rceil ) \big)\big\}\nn\\
&\quad \le C \lambda^{-\gamma q_2} \max \big\{ \rho \big( B (\lceil \Lambda' \boldsymbol{t} \rceil - \lceil B^{-1} \Lambda \boldsymbol{u} \rceil ) \big), 1 \big\},\label{ineq}
\end{align*}
where $C$ does not depend on $\boldsymbol{t},\boldsymbol{u} \in \R^2$. Then in view of the domination
$|\tilde g_\lambda(\boldsymbol{u})| \le C \tilde G_\lambda(\boldsymbol{u})$, $\mbu \in \R^2$, and
\eqref{tildeG}, the $L^2$-convergence in \eqref{L2conv} follows by Pratt's lemma from the convergence of norms
$\| \tilde G_\lambda \|^2 = \int_{(\boldsymbol{0},\boldsymbol{x}]\times
(\boldsymbol{0},\boldsymbol{x}]} (\rho^{-1} \star \rho^{-1} ) ( \Lambda^{-1}B\Lambda'(\boldsymbol{t}-\boldsymbol{s})) \d \boldsymbol{t} \d \boldsymbol{s}   \to \| \tilde G\|^2 $.
Indeed, after such a change of variables in the last integral that $b_{22} (t_2-s_2) = b_{22} (t'_2-s'_2)-\lambda^{1-\gamma} b_{21} (t_1-s_1)$
we see that
\begin{eqnarray} 
\| \tilde G_\lambda \|^2 &=& \int_{\R^4} (\rho^{-1} \star \rho^{-1} ) \big( \lambda^{1-\gamma \frac{q_2}{q_1} } \frac{\operatorname{det}(B)}{b_{22}} (t_1-s_1) + \lambda^{\gamma(1-\frac{q_2}{q_1})} b_{12} (t'_2-s'_2), b_{22}(t'_2-s'_2) \big)\nn \\
&&\qquad \times \1 \big(t_1\in (0,x_1], \, -\lambda^{1-\gamma} \frac{b_{21}}{b_{22}} t_1+ t'_2 \in (0,x_2]  \big)\nn \\
&&\qquad \times \1 \big(s_1\in (0,x_1], \, -\lambda^{1-\gamma}
\frac{b_{21}}{b_{22}} s_1+ s'_2 \in (0,x_2]  \big) \d t_1 \d t'_2 \d s_1 \d s'_2\nn \\
&\to&  \int_{(\boldsymbol{0},\boldsymbol{x}] \times (\boldsymbol{0},\boldsymbol{x}]} (\rho^{-1} \star \rho^{-1}) ( 0,b_{22}(t'_2-s'_2) )
\d t_1 \d t'_2 \d s_1 \d s'_2 = \| \tilde G \|^2 \label{Gnorm}
\end{eqnarray}
by the dominated convergence theorem
using the continuity of $(\rho^{-1} \star \rho^{-1})(\boldsymbol{t})$  and
$(\rho^{-1} \star \rho^{-1}) (\boldsymbol{t}) \le
C|t_2|^{-\tilde q_2}$ for $t_2 \neq 0$ with $\tilde q_2 = q_2 (2-Q)<1$, see  \eqref{def:conv}, \eqref{tilderho}.

\smallskip

\noi \underline{Case (i) and $\gamma = 1$ or $V^X_0 = \tilde V_{20}$.} Set $H(\gamma ) =  \frac{5}{2} + \frac{q_2}{2q_1} - q_2 =
1 + \tilde H_2 $. The proof is  similar to that in the case (i), $\gamma > 1 $ above.
We use Proposition \ref{prop} with $A = B^{-1}$, $\Lambda' =  \operatorname{diag}(\lambda, \lambda)$ and
$\Lambda = \operatorname{diag}(\lambda^{\frac{q_2}{q_1}}, \lambda)$. Accordingly, we need to prove $\|\tilde g_\lambda - \tilde h_{20}\| \to 0$, where
$$
\tilde h_{20} (\boldsymbol{u}) := 
\int_{(\boldsymbol{0},\boldsymbol{x}]} a_\infty (\tilde B_{20} \boldsymbol{t} - \boldsymbol{u}) \d \boldsymbol{t}, \quad \boldsymbol{u} \in \R^2,
$$
and $\tilde g_\lambda $ is defined as in \eqref{tildeg} with $\gamma =1 $. Note that now \eqref{point}  must be replaced by
\begin{equation*}
\Lambda^{-1} B (\lceil \Lambda' \boldsymbol{t} \rceil - \lceil B^{-1} \Lambda \boldsymbol{u} \rceil) \to \tilde B_{20} \boldsymbol{t} - \boldsymbol{u},
\end{equation*}
leading to
\begin{equation*}
\lambda^{q_2}
b ( \lceil \Lambda' \boldsymbol{t} \rceil - \lceil B^{-1} \Lambda \boldsymbol{u} \rceil ) \to
a_\infty( \tilde B_{20} \boldsymbol{t} - \boldsymbol{u})
\end{equation*}
for $\tilde B_{20} \boldsymbol{t} - \boldsymbol{u} \neq \boldsymbol{0}$.
Therefore, $\tilde g_\lambda$ converges to $\tilde h_{20}$ point-wise.  To prove the $L^2$-convergence
use Pratt's lemma as in the case $\gamma > 1 $ above, with $\tilde G_\lambda$, $\tilde G $ defined
as in \eqref{tildeG} with $\gamma =1 $ and $\tilde B_{22}$ replaced by $\tilde B_{20}$. Then
$\tilde G_\lambda(\boldsymbol{u}) \to \tilde G(\boldsymbol{u})$ for all $\mbu \in \R^2$ as in \eqref{tildeG} and
\begin{align*}
\| \tilde G_\lambda \|^2 &= \int_{(\boldsymbol{0}, \boldsymbol{x}] \times (\boldsymbol{0},\boldsymbol{x}]} (\rho^{-1} \star \rho^{-1}) ( \Lambda^{-1} B \Lambda' (\boldsymbol{t}-\boldsymbol{s}) ) \d \boldsymbol{t} \d \boldsymbol{s}\\
&=\int_{(\boldsymbol{0}, \boldsymbol{x}]\times(\boldsymbol{0},\boldsymbol{x}]} (\rho^{-1} \star \rho^{-1}) \big( \lambda^{1-\frac{q_2}{q_1}} \big( b_{11}(t_1-s_1) + b_{12} (t_2-s_2) \big), b_{21}(t_1-s_1)+b_{22}(t_2-s_2) \big) \d \boldsymbol{t} \d \boldsymbol{s}\\
&\to \int_{(\boldsymbol{0}, \boldsymbol{x}]\times(\boldsymbol{0},\boldsymbol{x}]} (\rho^{-1} \star \rho^{-1}) (\tilde B_{20}(\boldsymbol{t}-\boldsymbol{s})) \d \boldsymbol{t} \d \boldsymbol{s} = \| \tilde G \|^2
\end{align*}
follows similarly to \eqref{Gnorm}.

\smallskip

\noi \underline{Case (i) and $\gamma < 1$ or $V^X_- = \tilde V_{21}$.}
Set $H(\gamma ) = \gamma + \frac{3}{2} + \frac{q_2}{2q_1} - q_2 = \gamma + \tilde H_2$. The proof
proceeds similarly as above with $A = B^{-1}$, $\Lambda' =  \operatorname{diag}(\lambda, \lambda^\gamma)$
and $\Lambda = \operatorname{diag}(\lambda^{\frac{q_2}{q_1}}, \lambda)$. Then
\begin{align*}
\tilde g_\lambda(\boldsymbol{u})
&=
\int_{\R^2}  \lambda^{q_2}
b ( \lceil \Lambda' \boldsymbol{t} \rceil - \lceil B^{-1}  \Lambda \boldsymbol{u} \rceil ) \1 (
\Lambda' \boldsymbol{t} \in
(\boldsymbol{0}, \lfloor \Lambda' \boldsymbol{x} \rfloor]
) \d \boldsymbol{t} \\
&\to 
\int_{(\boldsymbol{0},\boldsymbol{x}]} a_\infty(\tilde B_{21} \boldsymbol{t}-\boldsymbol{u}) \d \boldsymbol{t} =: \tilde h_{21} (\boldsymbol{u}), \quad \mbu \in \R^2,
\end{align*}
in view of $\Lambda^{-1} B (\lceil \Lambda' \boldsymbol{t} \rceil - \lceil B^{-1} \Lambda \boldsymbol{u} \rceil) \to
\tilde B_{21} \boldsymbol{t} - \boldsymbol{u} $ and
$\lambda^{q_2}
b (\lceil \Lambda' \boldsymbol{t} \rceil - \lceil B^{-1} \Lambda \boldsymbol{u} \rceil ) \to a_\infty(\tilde B_{21} \boldsymbol{t} - \boldsymbol{u})$ for
$\tilde B_{21}\boldsymbol{t}-\boldsymbol{u} \neq \boldsymbol{0}$.  The proof of $\|\tilde g_\lambda - \tilde h_{21}\| \to 0$
using Pratt's lemma also follows similarly as above, with
$\tilde G_\lambda(\boldsymbol{u}) := \int_{(\boldsymbol{0},\boldsymbol{x}]} \rho(\Lambda^{-1}B\Lambda' \boldsymbol{t} - \boldsymbol{u})^{-1} \d \boldsymbol{t}
\to \int_{(\boldsymbol{0},\boldsymbol{x}]} \rho(\tilde B_{21} \boldsymbol{t}-\boldsymbol{u})^{-1} \d \boldsymbol{t}
=: \tilde G(\boldsymbol{u})$ for all $\mbu \in \R^2$ and
\begin{align*}
\| \tilde G_\lambda \|^2 &= \int_{\R^4} ( \rho^{-1} \star \rho^{-1} ) \big( \lambda^{1-\frac{q_2}{q_1}} b_{11} (t'_1-s'_1) - \lambda^{\gamma-\frac{q_2}{q_1}}  \frac{\operatorname{det}(B)}{b_{21}} (t_2-s_2), b_{21} (t'_1-s'_1) \big)\\
&\qquad \times \1 \big( t'_1-\lambda^{\gamma-1} \frac{b_{22}}{b_{21}}t_2 \in (0,x_1], \, t_2 \in (0,x_2] \big)\\
&\qquad \times \1 \big( s'_1-\lambda^{\gamma-1} \frac{b_{22}}{b_{21}}t_2 \in (0,x_1], \, s_2 \in (0,x_2] \big) \d t'_1 \d t_2 \d s'_1 \d s_2 \\
&\to \int_{(\boldsymbol{0},\boldsymbol{x}] \times (\boldsymbol{0},\boldsymbol{x}]} ( \rho^{-1} \star \rho^{-1} )(0,b_{21}(t'_1-s'_1)) \d t'_1 \d t_2 \d s'_1 \d s_2= \| \tilde G \|^2
\end{align*}
as in \eqref{Gnorm}.

\smallskip

\noi \underline{Case (ii) and $\gamma > \frac{q_1}{q_2}$ or $V^X_+ = \tilde V_{22}$.} Set $H(\gamma) = 1 + \gamma \tilde H_2$.
The proof of $\|\tilde g_\lambda - \tilde h_{22}\| \to 0$ is completely analogous to that
in Case (i), $\ga > 1 $, with the same $\tilde g_\la$, $\Lambda'$,
$\Lambda$, $\tilde h_{22}$ as in \eqref{tildeg}, \eqref{h22}  using the fact that
\begin{align*}
\Lambda^{-1} B \Lambda' =
\left[\begin{array}{rr}
\lambda^{1-\gamma \frac{q_2}{q_1}}b_{11}&\lambda^{\gamma(1-\frac{q_2}{q_1})}b_{12}\\
0&b_{22}
\end{array}\right] \to \tilde B_{22}.
\end{align*}

\smallskip

\noi \underline{Case (ii) and $\gamma = \frac{q_1}{q_2}$ or $V^X_0 = \tilde V_{00}$.} Set $H(\gamma ) =  \frac{3}{2}  + \frac{3q_1}{2q_2} - q_1
= \tilde H_1+\frac{q_1}{q_2}=1+\frac{q_1}{q_2}\tilde H_2, $
$A = B^{-1}$, $\Lambda = \Lambda' = \operatorname{diag}(\lambda,\allowbreak \lambda^{\frac{q_1}{q_2}})$.
The proof of $\|\tilde g_\lambda - \tilde h_{00}\| \to 0$ with
$\tilde h_{00}(\boldsymbol{u}) := 
\int_{(\boldsymbol{0},\boldsymbol{x}]} a_\infty (\tilde B_{00}\boldsymbol{t}-\boldsymbol{u}) \d \boldsymbol{t}$, $\mbu \in \R^2$,
follows similar lines as in the other cases. The point-wise convergence $\tilde g_\lambda \to  \tilde h_{00}$ uses
$\Lambda^{-1}B(\lceil \Lambda' \boldsymbol{t} \rceil- \lceil B^{-1} \Lambda \boldsymbol{u} \rceil) \to
\tilde B_{00}\boldsymbol{t} - \boldsymbol{u} $ and
$\lambda^{q_1}
	b (\lceil \Lambda' \boldsymbol{t} \rceil - \lceil B^{-1} \Lambda \boldsymbol{u} \rceil ) \to
a_\infty(\tilde B_{00}\boldsymbol{t} - \boldsymbol{u}) $ for
$\tilde B_{00}\boldsymbol{t} - \boldsymbol{u} \neq \boldsymbol{0}$.
The $L^2$-convergence can be verified using Pratt's lemma with the dominating function
$\tilde G_\lambda (\boldsymbol{u}) :=\int_{(\boldsymbol{0},\boldsymbol{x}]} \rho ( \Lambda^{-1} B \Lambda' \boldsymbol{t} - \boldsymbol{u} )^{-1} \d \boldsymbol{t} \to \int_{(\boldsymbol{0},\boldsymbol{x}]} \rho(\tilde B_{00}\boldsymbol{t} - \boldsymbol{u} )^{-1} \d \boldsymbol{t} =: \tilde G(\boldsymbol{u})$, $\mbu \in \R^2$,
satisfying
\begin{align*}
	\| \tilde G_\lambda \|^2 &=\int_{\R^4} (\rho^{-1} \star \rho^{-1}) \big( b_{11} (t'_1-s'_1), b_{22} (t_2-s_2) \big) \\
	&\qquad \times \1 \big( t'_1 - \lambda^{\frac{q_1}{q_2}-1} \frac{b_{12}}{b_{11}} t_2 \in (0,x_1], \, t_2 \in (0,x_2] \big)\\
	&\qquad \times \1 \big( s'_1 - \lambda^{\frac{q_1}{q_2}-1} \frac{b_{12}}{b_{11}} s_2 \in (0,x_1], \, s_2 \in (0,x_2] \big)\d t_1 \d t_2 \d s_1 \d s_2\\
	&\to \int_{(\boldsymbol{0},\boldsymbol{x}] \times (\boldsymbol{0},\boldsymbol{x}]} (\rho^{-1} \star \rho^{-1}) (b_{11}(t'_1-s'_1), b_{22}(t_2-s_2)) \d t'_1 \d t_2 \d s'_1 \d s_2 = \|\tilde G \|^2,
	\end{align*}
which follows from the dominated convergence theorem using $(\rho^{-1} \star \rho^{-1}) (\boldsymbol{t}) \le C \tilde \rho(\boldsymbol{t})^{-1}$, $\mbt \in \R^2_0$,
and the (local) integrability of the function $\tilde \rho^{-1}$ with $\frac{1}{\tilde q_1} + \frac{1}{\tilde q_2} =
\frac{Q}{2-Q} >1$, see
\eqref{def:conv}, \eqref{tilderho} and \eqref{intrho}.

We note that the above proof applies for all $q_1 <q_2 $ satisfying  $1 < Q < 2$, hence also in Cases (ii), $V^X_0 = \tilde V_{00}$ of Theorems
\ref{thmR12} and \ref{thmR22}.

\smallskip

\noi \underline{Case (ii) and $\gamma < \frac{q_1}{q_2}$ or $V^X_- = \tilde V_{11}$.}
 Set $H(\gamma ) =  \ga +\frac{3}{2}  +\frac{q_1}{2q_2} - q_1 = \ga + \tilde H_1$, $A = B^{-1}$, $\Lambda  = \operatorname{diag}(\la, \la^{\frac{q_1}{q_2}})$, $\Lambda'  = \operatorname{diag}(\la, \la^{\ga})$. Then
$\Lambda^{-1} B (\lceil \Lambda' \boldsymbol{t} \rceil - \lceil B^{-1} \Lambda \boldsymbol{u} \rceil) \to
\tilde B_{11}\boldsymbol{t} - \boldsymbol{u} $ and
$\lambda^{q_1}
b ( \lceil \Lambda' \boldsymbol{t}\rceil - \lceil B^{-1} \Lambda \boldsymbol{u} \rceil ) \to a_\infty(\tilde B_{11}\boldsymbol{t} - \boldsymbol{u})$ for
$\tilde B_{11} \boldsymbol{t}-\boldsymbol{u} \neq \boldsymbol{0}$. This leads to
$\tilde g_\lambda (\mbu) \to
h_{11}(\boldsymbol{u}) := 
\int_{(\boldsymbol{0},\boldsymbol{x}]} a_\infty(\tilde B_{11} \boldsymbol{t}-\boldsymbol{u}) \d \boldsymbol{t}$ for all $\mbu \in \R^2 $.  The required convergence $\|\tilde g_\la - \tilde h_{11}\| \to 0$ follows similarly as in the other cases
using $\tilde G_\lambda (\boldsymbol{u}) := \int_{(\boldsymbol{0},\boldsymbol{x}]} \rho(\Lambda^{-1} B \Lambda' \boldsymbol{t} - \boldsymbol{u})^{-1} \d \boldsymbol{t} \to \int_{(\boldsymbol{0},\boldsymbol{x}]} \rho(\tilde B_{11} \boldsymbol{t}-\boldsymbol{u})^{-1} \d \boldsymbol{t} =: \tilde G(\boldsymbol{u})
$ for all $\boldsymbol{u} \in \R^2$ and
\begin{eqnarray*}
\|\tilde G_\lambda \|^2
&=& \int_{\R^4} (\rho^{-1} \star \rho^{-1}) \big(b_{11} (t'_1-s'_1),
\lambda^{\gamma-\frac{q_1}{q_2}} b_{22} (t_2-s_2) \big)\\
&&\qquad \times \1 \big( t'_1 - \lambda^{\gamma-1} \frac{b_{12}}{b_{11}} t_2 \in (0,x_1], \, t_2 \in (0,x_2] \big)\\
&&\qquad \times \1 \big( s'_1 - \lambda^{\gamma-1} \frac{b_{12}}{b_{11}} s_2 \in (0,x_1], \, s_2 \in (0,x_2] \big)
\d t^\prime_1 \d t_2 \d s'_1 \d s_2 \\
&\to&\int_{(\boldsymbol{0},\boldsymbol{x}] \times (\boldsymbol{0}, \boldsymbol{x}]}
(\rho^{-1} \star \rho^{-1}) (b_{11}(t^\prime_1-s^\prime_1), 0) \d t^\prime_1 \d t_2 \d s^\prime_1 \d s_2 = \| \tilde G \|^2
\end{eqnarray*}
which follows from the dominated convergence theorem using $(\rho^{-1} \star \rho^{-1})(\boldsymbol{t}) \le C |t_1|^{- \tilde q_1}$ for all $t_1 \neq 0$ with $\tilde q_1 <1$.

Note that the above proof applies for all $q_1 <q_2 $ satisfying  $\tilde q_1 <1$ or $\tilde Q_2 > 1$
hence also in Case (ii), $V^X_- = \tilde V_{11}$ of Theorem
\ref{thmR12}.
Theorem \ref{thmR11} is proved.
\end{proof}

\begin{proof}[Proof of Theorem \ref{thmR12}.]
\underline{Case (i) and  $\gamma > 1$ or $V^X_+ = V_{11}$.} 
Set  $H(\gamma ) =  \frac{3}{2} + \frac{q_1}{q_2} - q_1 + \frac{\gamma}{2} = H_1 + \frac{\gamma}{2}$.
Rewrite $g_\lambda (\mbs)$ as
\begin{align} \label{gla}
g_\lambda (\boldsymbol{s})
&= \lambda^{-H(\gamma)} \int_{\R^2} b(\lceil \boldsymbol{t} \rceil) \1 ( \lceil \boldsymbol{t} \rceil + \mbs
\in (0, \lfloor \lambda x_1 \rfloor ] \times (0, \lfloor \lambda^\gamma x_2 \rfloor] ) \d \mbt \nn \\
&= \lambda^{1+\frac{q_1}{q_2}-H(\gamma)}
\int_{\R^2} b(\lceil B^{-1}  \Lambda' \boldsymbol{t} \rceil ) \1 ( \lceil B^{-1} \Lambda' \boldsymbol{t} \rceil+ \boldsymbol{s} \in (0, \lfloor \lambda x_1 \rfloor ] \times (0, \lfloor \lambda^\gamma x_2 \rfloor] ) \d \boldsymbol{t}, \quad \mbs \in \Z^2,
\end{align}
where $\Lambda' := \operatorname{diag} (\la, \la^{\frac{q_1}{q_2}})$. We use Proposition \ref{prop} with
\begin{eqnarray} \label{AL}
&A := \left[
\begin{array}{cc}
1&0\\
- \frac{b_{21}}{b_{22}}&1
\end{array} \right], \quad
\Lambda := 
\left[
\begin{array}{ll}
\lambda&0\\
0&\lambda^\gamma
\end{array} \right].
\end{eqnarray}
Note $\Lambda^{-1} B^{-1} \Lambda' \to B_{11}$, $\Lambda^{-1} A  \Lambda \to I$ since $\gamma > 1 $.
Then, according to the definition \eqref{def:tildegn},
$\tilde g_\lambda (\boldsymbol{u}) = \int_{\R^2} \tilde b_\lambda (\boldsymbol{u}, \boldsymbol{t}) \d \boldsymbol{t},$
where
\begin{align}\label{tb11}
\tilde b_\lambda (\boldsymbol{u},\boldsymbol{t}) := \lambda^{q_1} b(\lceil B^{-1} \Lambda' \boldsymbol{t} \rceil) \1 ( \lceil B^{-1}  \Lambda' \boldsymbol{t} \rceil + \lceil A \Lambda \boldsymbol{u} \rceil \in (\boldsymbol{0},\lfloor \Lambda \boldsymbol{x}\rfloor ] ) \to a_\infty (\boldsymbol{t}) \1 ( B_{11} \boldsymbol{t} + \boldsymbol{u} \in (\boldsymbol{0},\boldsymbol{x}] )
\end{align}
for all $\boldsymbol{u}, \boldsymbol{t} \in \R^2$ such that $\boldsymbol{t} \neq \boldsymbol{0}$, $b_{22} t_1 + u_1 \not \in \{0, x_1\}$, $u_2 \not \in \{0, x_2\}$. It suffices to prove that the following point-wise convergence also holds
\begin{equation}\label{th11}
\tilde g_\lambda (\boldsymbol{u}) \to h_{11} (\boldsymbol{u}) := \int_{\R^2} a_\infty (\boldsymbol{t}) \1 (B_{11} \boldsymbol{t} + \boldsymbol{u} \in (\boldsymbol{0},\boldsymbol{x}] ) \d \boldsymbol{t} \quad \text{in } L^2(\R^2).
\end{equation}
To show \eqref{th11}, decompose $\tilde g_{\lambda}(\boldsymbol{u}) =  \tilde g_{\lambda,0}(\boldsymbol{u}) + \tilde g_{\lambda,1}(\boldsymbol{u})$ with $\tilde g_{\lambda,j}(\boldsymbol{u}) = \int_{\R^2} \tilde b_{\lambda,j}(\boldsymbol{u}, \boldsymbol{t}) \d \boldsymbol{t}$ given by
\begin{align*}
\tilde b_{\lambda,0} (\boldsymbol{u},\boldsymbol{t}) := \tilde b_{\lambda} (\boldsymbol{u},\boldsymbol{t}) \1 \big( | t_2 | \ge \lambda^{1-\frac{q_1}{q_2}} \big), \quad
\tilde b_{\lambda,1} (\boldsymbol{u},\boldsymbol{t}) := \tilde b_{\lambda} (\boldsymbol{u},\boldsymbol{t}) \1 \big( | t_2 | < \lambda^{1-\frac{q_1}{q_2}} \big),  \quad \boldsymbol{u}, \boldsymbol{t} \in \R^2.
\end{align*}
Then \eqref{th11} follows from
\begin{equation}\label{th112}
\| \tilde g_{\lambda,1} - h_{11} \| \to 0 \quad \text{and} \quad  \| \tilde g_{\lambda,0} \| \to 0.
\end{equation}
The first relation in \eqref{th112} follows from \eqref{tb11} and the dominated convergence theorem, as follows. To justify the domination, combine   $|b(\boldsymbol{s})| \le C \max \{\rho(B\boldsymbol{s}),1 \}^{-1}$, $\boldsymbol{s} \in \Z^2$, and
$\rho(\boldsymbol{t}) \le C \lambda^{- q_1} \max \{ \rho ( B \lceil B^{-1} \Lambda' \boldsymbol{t} \rceil ), 1 \}$, $\boldsymbol{t} \in \R^2$,
to get $\lambda^{q_1} |b(\lceil B^{-1} \Lambda' \boldsymbol{t} \rceil)| \le C \rho(\mbt)^{-1}$, $\mbt \in \R^2$, $\lambda > 1$.
Also note  that
\begin{align*}
&\1 \big( |t_2| < \lambda^{1-\frac{q_1}{q_2}}, \, \lceil B^{-1}  \Lambda' \boldsymbol{t} \rceil + \lceil A \Lambda \boldsymbol{u} \rceil \in (\boldsymbol{0},\lfloor \Lambda \boldsymbol{x}\rfloor ] \big)\\
&\qquad\le \1 \big( |t_2| < \lambda^{1-\frac{q_1}{q_2}}, \,
B^{-1}  \Lambda' \boldsymbol{t} + A \Lambda \boldsymbol{u} \in (-{\bf 2}, \Lambda \boldsymbol{x}]\big) \\
&\qquad \le \1 \big( |t_2| < \lambda^{1-\frac{q_1}{q_2}}, \,
b_{22} t_1 - \lambda^{\frac{q_1}{q_2}-1}  b_{12} t_2 + u_1 \in (-\frac{2}{\lambda}, x_1 ], \\
&\qquad \qquad - \lambda^{1-\gamma} \frac{b_{21}}{b_{22}} ( b_{22} t_1 - \lambda^{\frac{q_1}{q_2}-1} b_{12} t_2 + u_1) + \lambda^{\frac{q_1}{q_2}-\gamma} \frac{\operatorname{det}(B)}{b_{22}} t_2 + u_2 \in (-\frac{2}{\lambda^\gamma}, x_2] \big)\\
&\qquad \le \1 \big( | b_{22} t_1 + u_1 | \le C_1, \, |u_2| \le C_2 \big)
\end{align*}
for  some $C_1, C_2 >0$  independent of $\boldsymbol{t}, \boldsymbol{u} \in \R^2 $ and $\lambda >0$. Thus, $|\tilde g_{\lambda,1}(\mbu)| \le \bar g (\mbu)$, where the dominating function $\bar g(\mbu) :=  C \1(|u_2| \le C_2) \int_{\R^2} \rho(\mbt)^{-1} \1( | b_{22} t_1 + u_1 | \le C_1) \d \mbt
\le  C \1(|u_2| \le C_2) \int_{\R} |t_1|^{\frac{q_1}{q_2}- q_1} \1( | b_{22} t_1 + u_1 | \le C_1) \d t_1 $, $\mbu \in \R^2$, satisfies $\|\bar g\| <\infty$ since $\tilde Q_1 < 1$, proving the first relation in \eqref{th112}. The second relation in \eqref{th112} follows by
Minkowski's inequality:
\begin{align*}
\| \tilde g_{\lambda,0} \| &\le C \int_{\R^2} \Big( \int_{\R^2} \rho(t_1-\frac{1}{b_{22}} u_1,t_2)^{-2} \1 \big( | t_2| \ge \lambda^{1-\frac{q_1}{q_2}} \big)\\
&\qquad \times \1 \big( | b_{22}t_1- \lambda^{\frac{q_1}{q_2}-1}  b_{12} t_2 | \le C_1, \, | \lambda^{\frac{q_1}{q_2}-\gamma} \frac{\operatorname{det}(B)}{b_{22}} t_2 + u_2 | \le C_2 \big) \d \boldsymbol{u} \Big)^{\frac{1}{2}} \d \boldsymbol{t}\\
&= C \int_{\R^2} \Big( \int_{\R} \rho(u_1,t_2)^{-2} \1 \big( | t_2 | \ge \lambda^{1-\frac{q_1}{q_2}}, \, | b_{22} t_1 - \lambda^{\frac{q_1}{q_2}-1} b_{12} t_2 | \le C_1 \big) \d u_1 \Big)^{\frac{1}{2}} \d \boldsymbol{t}\\
&= C \int_{\R^2} |t_2|^{q_2(\frac{1}{2q_1}-1)} \1 \big( | t_2 | \ge \lambda^{1-\frac{q_1}{q_2}}, \, | b_{22} t_1 - \lambda^{\frac{q_1}{q_2}-1} b_{12} t_2 | \le C_1 \big) \d \boldsymbol{t} \\
&= C\int_{\R} |t_2|^{q_2(\frac{1}{2q_1}-1)} \1 \big( | t_2 | \ge \lambda^{1-\frac{q_1}{q_2}}\big) \d t_2 = o(1)
\end{align*}
since $\frac{1}{2q_1} < \tilde Q_1 < 1$. This proves \eqref{th112} and  \eqref{th11}.

We note that the above argument  applies to the proof of the limit $V^X_+ = V_{11}$ in both
Cases (i) and (ii)  of Theorems \ref{thmR12} and  \ref{thmR22}
as well,  including Case (ii) and $\frac{q_1}{q_2} < \gamma \le 1 $, with the difference
that in the latter case we make the change of variable $\mbt \to B^{-1} \Lambda' \mbt $ in \eqref{gla} with
$\Lambda'  :=   \operatorname{diag}(\lambda^\gamma,\lambda^{\gamma \frac{q_1}{q_2}})$.

\smallskip

\noi \underline{Case (i) and  $\gamma = 1$ or $V^X_+ = V_{01}$.} Set $H(\gamma) = H_1 + \frac{1}{2}$.
The proof proceeds as in Case (i), $\gamma > 1 $ above, by writing
$\tilde g_\lambda (\boldsymbol{u}) = \int_{\R^2} \tilde b_\lambda (\boldsymbol{u}, \boldsymbol{t}) \d \boldsymbol{t}$ with $A$ as in \eqref{AL}, $\Lambda := \operatorname{diag}(\lambda,\lambda)$,  $\Lambda' := \operatorname{diag}(\lambda,\lambda^{\frac{q_1}{q_2}})$ in
$$
\tilde b_\lambda (\boldsymbol{u},\boldsymbol{t}) := \lambda^{q_1} b(\lceil B^{-1} \Lambda' \boldsymbol{t} \rceil) \1 ( \lceil B^{-1}  \Lambda' \boldsymbol{t} \rceil + \lceil \Lambda \boldsymbol{u} \rceil \in (\boldsymbol{0},\lfloor \Lambda \boldsymbol{x}\rfloor ] ) \to a_\infty (\boldsymbol{t}) \1 ( B_{01} \boldsymbol{t} + \boldsymbol{u} \in (\boldsymbol{0},\boldsymbol{x}] )
$$
for all $\mbu, \mbt \in \R^2$ such that $\mbt \neq \boldsymbol{0}$, $b_{22} t_1 + u_1 \not \in \{0, x_1 \}$, $-b_{21} t_1 + u_2 \not \in \{0, x_2\}$, c.f.\ \eqref{tb11}.
The details of the convergence $\tilde g_\lambda (\boldsymbol{u}) \to h_{01} (\boldsymbol{u}) := \int_{\R^2} a_\infty (\boldsymbol{t}) \1 \big(B_{01} \boldsymbol{t} + \boldsymbol{u} \in (\boldsymbol{0},\boldsymbol{x}] \big) \d \boldsymbol{t} $, $\mbu \in \R^2$, in $L^2(\R^2)$
are similar as above and
omitted.

We note that the above argument  applies to the proof of $V^X_0 = V_{01}$ in
Case (i) of Theorem \ref{thmR22}.

\smallskip

\noi \underline{Case (i) and  $\gamma < 1$ or $V^X_- = V_{21}$.}
Set $H(\gamma ) =  \frac{1}{2} +  \gamma (\frac{3}{2} +  \frac{q_1}{q_2} - q_1) = \frac{1}{2} + \gamma  H_1 $.
Use Proposition \ref{prop} with
\begin{eqnarray*}
&A := \left[
\begin{array}{cc}
1&- \frac{b_{22}}{b_{21}} \\
0&1
\end{array} \right], \quad
\Lambda := 
\left[
\begin{array}{ll}
\lambda&0\\
0&\lambda^\gamma
\end{array} \right].
\end{eqnarray*}
Similarly to \eqref{tb11},
$\tilde g_\lambda (\boldsymbol{u}) = \int_{\R^2} \tilde b_\lambda (\boldsymbol{u}, \boldsymbol{t}) \d \boldsymbol{t}$,
where
$$
\tilde b_\lambda (\boldsymbol{u},\boldsymbol{t}) := \lambda^{\gamma q_1} b(\lceil B^{-1} \Lambda' \boldsymbol{t} \rceil) \1 ( \lceil B^{-1}  \Lambda' \boldsymbol{t} \rceil + \lceil \Lambda \boldsymbol{u} \rceil \in (\boldsymbol{0},\lfloor \Lambda \boldsymbol{x}\rfloor ] ) \to a_\infty (\boldsymbol{t}) \1 (B_{21} \boldsymbol{t} + \boldsymbol{u} \in (\boldsymbol{0},\boldsymbol{x}] )
$$
for all $\boldsymbol{u}, \boldsymbol{t} \in \R^2$ such that $\boldsymbol{t} \neq \boldsymbol{0}$, $u_1 \not \in \{0, x_1\}$, $-b_{21} t_1 + u_2 \not \in \{0, x_2\}$. It suffices to prove
$\tilde g_\lambda (\boldsymbol{u}) \ \to \ h_{21} (\boldsymbol{u}) := \int_{\R^2} a_\infty (\boldsymbol{t})
\1 \big(B_{21} \boldsymbol{t} + \boldsymbol{u} \in (\boldsymbol{0},\boldsymbol{x}] \big) \d \boldsymbol{t}$ in $L^2(\R^2)$
or
\begin{equation}\label{th212}
\| \tilde g_{\lambda,1} - h_{21} \| \to 0  \quad \text{and} \quad \| \tilde g_{\lambda,0} \| \to 0,
\end{equation}
where  $\tilde g_{\lambda,j}(\boldsymbol{u}) = \int_{\R^2} \tilde b_{\lambda,j}(\boldsymbol{u}, \boldsymbol{t}) \d \boldsymbol{t}$ \ and
\begin{align*}
\tilde b_{\lambda,0} (\boldsymbol{u},\boldsymbol{t}) := \tilde b_{\lambda} (\boldsymbol{u},\boldsymbol{t}) \1 \big( | t_2 | \ge \lambda^{\gamma(1-\frac{q_1}{q_2})} \big), \quad
\tilde b_{\lambda,1} (\boldsymbol{u},\boldsymbol{t}) := \tilde b_{\lambda} (\boldsymbol{u},\boldsymbol{t}) \1 \big( | t_2 | < \lambda^{\gamma(1-\frac{q_1}{q_2})} \big), \quad \boldsymbol{u}, \boldsymbol{t} \in \R^2.
\end{align*}
The proof of \eqref{th212} is similar to that of \eqref{th112} and omitted. This proves Case (i).
The above proofs also included Case (ii) of  Theorem \ref{thmR12}.
Theorem \ref{thmR12} is proved.
\end{proof}

\begin{proof}[Proof of Theorem \ref{thmR22}.]
\underline{Case (ii) and  $\gamma < \frac{q_1}{q_2}$ or $V^X_- =  V_{22}$.}
Set $H(\gamma ) =  \frac{1}{2} +  \gamma (\frac{3}{2} +  \frac{q_2}{q_1} - q_2) = \frac{1}{2} + \gamma  H_2 $.
The proof is similar to that in the case (i) of Theorem \ref{thmR12} using
\begin{eqnarray*}
&A := \left[
\begin{array}{cc}
1&- \frac{b_{12}}{b_{11}}\\
0&1
\end{array} \right], \quad
\Lambda := 
\left[
\begin{array}{ll}
\lambda&0\\
0&\lambda^\gamma
\end{array} \right], \quad \Lambda' := 
\left[
\begin{array}{cc}
\lambda^{\gamma \frac{q_2}{q_1}}&0\\
0&\lambda^\gamma
\end{array} \right].
\end{eqnarray*}
Accordingly,
$\tilde g_\lambda (\boldsymbol{u}) = \int_{\R^2} \tilde b_\lambda (\boldsymbol{u}, \boldsymbol{t}) \d \boldsymbol{t}$,
where
$$
\tilde b_\lambda (\boldsymbol{u},\boldsymbol{t}) := \lambda^{\gamma q_2}
b(\lceil B^{-1} \Lambda' \boldsymbol{t} \rceil) \1 ( \lceil B^{-1}  \Lambda' \boldsymbol{t} \rceil + \lceil \Lambda \boldsymbol{u} \rceil \in (\boldsymbol{0},\lfloor \Lambda \boldsymbol{x}\rfloor ] ) \to a_\infty (\boldsymbol{t}) \1 (B_{22} \boldsymbol{t} + \boldsymbol{u} \in (\boldsymbol{0},\boldsymbol{x}] )
$$
for all $\boldsymbol{u}, \boldsymbol{t} \in \R^2$ such that $\boldsymbol{t} \neq \boldsymbol{0}$, $u_1 \not \in \{0, x_1\}$, $b_{11} t_2 + u_2 \not \in \{0, x_2\}$. It suffices to prove
$\tilde g_\lambda (\boldsymbol{u}) \ \to \ h_{22} (\boldsymbol{u}) := \int_{\R^2} a_\infty (\boldsymbol{t})
\1 ( B_{22} \boldsymbol{t} + \boldsymbol{u} \in (\boldsymbol{0},\boldsymbol{x}] ) \d \boldsymbol{t}$ in $L^2(\R^2)$
or
\begin{equation}\label{th222}
\| \tilde g_{\lambda,1} - h_{22} \| \to 0  \quad \text{and} \quad \| \tilde g_{\lambda,0} \| \to 0,
\end{equation}
where  $\tilde g_{\lambda,j}(\boldsymbol{u}) = \int_{\R^2} \tilde b_{\lambda,j}(\boldsymbol{u}, \boldsymbol{t}) \d \boldsymbol{t}$ \ and
\begin{align*}
\tilde b_{\lambda,0} (\boldsymbol{u},\boldsymbol{t}) := \tilde b_{\lambda} (\boldsymbol{u},\boldsymbol{t}) \1 \big( | t_1 | \ge
\lambda^{\gamma \frac{q_2}{q_1} -1} \big), \quad
\tilde b_{\lambda,1} (\boldsymbol{u},\boldsymbol{t}) :=
\tilde b_{\lambda} (\boldsymbol{u},\boldsymbol{t}) \1 \big( | t_1 |
< \lambda^{\gamma \frac{q_2}{q_1}-1} \big), \quad \boldsymbol{u}, \boldsymbol{t} \in \R^2.
\end{align*}
The proof of \eqref{th222} using $\tilde Q_2 <1 $ is similar to that of \eqref{th112} and omitted.
The remaining cases of Theorem~\ref{thmR22} follow from Theorem \ref{thmR12}, thereby completing
the proof of Theorem~\ref{thmR22}.
\end{proof}

\begin{proof}[Proof of Theorem \ref{thmR1}.]
\underline{Case (i) and  $\gamma > 1$ or $V^X_+ = \tilde V_{02}$.} Set $H(\gamma) = 1 + \gamma \tilde H$. We follow the proof
of Theorem \ref{thmR11}, Case (i), $\gamma > 1$. Let $\Lambda := \operatorname{diag}(\lambda^\ga, \la^\ga)$, $\Lambda' :=  \operatorname{diag}(\lambda, \la^\ga)$. Then
\eqref{point} and \eqref{lab1} hold with $q_2 = q$ and $\tilde B_{22} $ replaced by $\tilde B_{02}$.
Then the convergence $\tilde g_\lambda(\mbu) \to \tilde h_{02}(\mbu) := \int_{\boldsymbol{0}, {\mbx}]} a_\infty (\tilde B_{02} \mbt - \mbu) \d \mbt $ in $L^2(\R^2)$
follows similarly by Pratt's lemma with  \eqref{Gnorm} replaced by
\begin{eqnarray*} 
\| \tilde G_\lambda \|^2 &=& \int_{\R^4} (\rho^{-1} \star \rho^{-1} ) \big( \lambda^{1-\gamma} \frac{\operatorname{det}(B)}{b_{22}} (t_1-s_1) + b_{12} (t'_2-s'_2), b_{22}(t'_2-s'_2) \big)\nn \\
&&\qquad \times \1 \big(t_1\in (0,x_1], \, -\lambda^{1-\gamma} \frac{b_{21}}{b_{22}} t_1+ t'_2 \in (0,x_2]  \big)\nn \\
&&\qquad \times \1 \big(s_1\in (0,x_1], \, -\lambda^{1-\gamma}
\frac{b_{21}}{b_{22}} s_1+ s'_2 \in (0,x_2]  \big) \d t_1 \d t'_2 \d s_1 \d s'_2\nn \\
&\to&  \int_{(\boldsymbol{0},\boldsymbol{x}] \times (\boldsymbol{0},\boldsymbol{x}]} (\rho^{-1} \star \rho^{-1}) (\tilde B_{22}(\mbt -\mbs) )
\d \mbt \d \mbs = \| \tilde G \|^2.  
\end{eqnarray*}
The proof in Case (i), $\gamma = 1$ and  $\gamma < 1$ is similar
to that of Theorem \ref{thmR11},  Case (i) 
and is omitted.

\smallskip

\noi \underline{Case (ii) and  $\gamma > 1$ or $V^X_+ =  V_{10}$.} Set $H(\gamma) = H + \frac{\gamma}{2}$.
We follow the proof of Theorem \ref{thmR12}, Case (i), $\gamma > 1$, with
$\Lambda' := \operatorname{diag} (\la, \la)$ and $A, \Lambda $ as in
\eqref{AL}. Then $\Lambda^{-1} B^{-1} \Lambda' \to B_{10}$ and the result follows
from $\|\tilde g_\lambda -  h_{10}\| \to 0 $, where
$h_{10}(\mbu) := \int_{\R^2} a_\infty (\mbt) \1( B_{10} \mbt + \mbu \in(\boldsymbol{0}, {\mbx}]) \d \mbt $.
Following the proof of \eqref{th11}, we decompose $\tilde g_{\lambda}(\boldsymbol{u}) =  \tilde g_{\lambda,0}(\boldsymbol{u}) + \tilde g_{\lambda,1}(\boldsymbol{u})$ with $\tilde g_{\lambda,j}(\boldsymbol{u}) = \int_{\R^2} \tilde b_{\lambda,j}(\boldsymbol{u}, \boldsymbol{t}) \d \boldsymbol{t}$ given by
\begin{equation*}
\tilde b_{\lambda,0} (\boldsymbol{u},\boldsymbol{t}) :=
\tilde b_{\lambda} (\boldsymbol{u},\boldsymbol{t}) \1 ( | t_2 | \ge \tilde C), \quad
\tilde b_{\lambda,1} (\boldsymbol{u},\boldsymbol{t}) := \tilde b_{\lambda} (\boldsymbol{u},\boldsymbol{t}) \1 ( | t_2 | < \tilde C ),
\end{equation*}
and $\tilde b_{\lambda} (\boldsymbol{u},\boldsymbol{t})$, $\mbu, \mbt \in \R^2$, as in \eqref{tb11},
where $\tilde C >0$ is a sufficiently large constant. It suffices to prove
\begin{equation} \label{limC}
\lim_{\tilde C \to \infty} \limsup_{\lambda \to \infty} \|\tilde g_{\lambda,1} - h_{10}\| = 0 \quad \text{and} \quad
\lim_{\tilde C \to \infty} \limsup_{\lambda \to \infty} \|\tilde g_{\lambda,0}\| = 0.
\end{equation}
The proof of \eqref{limC} mimics that of \eqref{th212} and we omit the details.
The remaining statements in Theorem~\ref{thmR1}, Case~(ii) also follow similarly to
the proof of Theorem~\ref{thmR12}, Case~(i). Theorem~\ref{thmR1} is proved.
\end{proof}

\section{Appendix}

\subsection{Generalized homogeneous functions}

Let $q_i >0$, $i=1,2$, with $Q := \frac{1}{q_1} + \frac{1}{q_2}$.

\begin{defn}  \label{defhomo} {\rm A measurable function $h: \R^2 \to \R$ is said to be:

\smallskip

\noi (i) {\it generalized homogeneous} if for all $\lambda >0$,
\begin{equation} \label{anishomo}
\lambda h(\lambda^{1/q_1}t_1, \lambda^{1/q_2} t_2) =  h({\mbt}), \quad {\mbt} \in \R^2_0. 
\end{equation}

\noi (ii) {\it generalized invariant}  if
$\lambda \mapsto h(\lambda^{1/q_1}t_1, \lambda^{1/q_2} t_2)$ is a constant function on $\R_+$ for
any ${\mbt} \in \R^2_0 $.
}

\end{defn}

\begin{prop} Any generalized homogeneous function $h$ can be represented as
\begin{equation}\label{hrho}
h(\mbt) = L(\mbt)/\rho(\mbt), \quad \mbt  \in \R^2_0,
\end{equation}
where $\rho(\mbt) = |t_1|^{q_1} +  |t_2|^{q_2}$, $\boldsymbol{t} \in \R^2$, and $L$ is a generalized invariant function. Moreover,
$L$ can be written as
\begin{equation} \label{Lpm}
L(\mbt) = L_{\operatorname{sign}(t_2)} ( t_1/\rho(\mbt)^{1/q_1} ), \quad t_2 \ne 0,
\end{equation}
where $L_\pm (z) := h(z, \pm (1 - |z|^{q_1})^{1/q_2})$, $z \in [-1,1]$.
\end{prop}
\begin{proof}
\eqref{hrho}  follows from \eqref{anishomo}, by taking $\lambda = 1/\rho(\mbt). $ Then
$L(\mbt)  := h( t_1/\rho(\mbt)^{1/q_1}, t_2/\rho(\mbt)^{1/q_2}) $ is a generalized invariant function. Whence,
\eqref{Lpm} follows since $  t_2/\rho(\mbt)^{1/q_2} = \operatorname{sign}(t_2) (1 -  |t_1/\rho(\mbt)^{1/q_1}|^{q_1})^{1/q_2} $.
\end{proof}

The notion of generalized homogeneous function was introduced in \cite{han1972}. The last paper also obtained
a representation of such functions different from \eqref{hrho}.
Note ${\mbt} \mapsto (t_1, \rho(\mbt)^{1/q_1}) $
is a 1-1 transformation of the upper half-plane $\{{\mbt} \in \R^2: t_2 \ge 0 \}$ onto itself.
Following \cite{dob1979}, the form in \eqref{hrho} will be called the {\it polar representation of $h$.}
The two  factors in  \eqref{hrho}, viz.,
$\rho^{-1}$ and $L$
are called the {\it radial}  and
{\it angular} functions, respectively.  Note that $h$ being strictly positive and continuous on $\R^2_0$ is equivalent
to $L_\pm$ both being strictly positive and continuous on $[-1,1]$ with $L_+(\pm 1) = L_-(\pm 1)$.

\subsection{Dependence axis }

\begin{defn} \label{depdef}
{\rm Let  $g \colon \R^2 \to \R $ be a measurable function. We say that a line passing through the origin and given by $\{ \mbt \in \R^2 : \mba \cdot \mbt = 0 \}$ with $\mba \in \R^2_0$
is the {\it dependence axis of $g$} if for all $\mbc \in \R^2_0$ such that $a_1 c_2  \ne c_1 a_2$,
\begin{eqnarray}
\liminf_{|\mbt| \to \infty, \, {\mbc}\cdot {\mbt} = 0} \frac{\log (1/|g({\mbt})|)}{\log |{\mbt}|}
 &>& \limsup_{|{\mbt}| \to \infty, \, {\mba}\cdot {\mbt} = 0} \frac{\log (1/|g({\mbt})|)}{\log |{\mbt}|}.
\end{eqnarray}
We say that a line $\{ \mbt \in \R^2 : \mba \cdot \mbt = 0 \}$ with $\mba \in \R^2_0$
is the  {\it dependence axis of
$g : \Z^2 \to \R$} if this line is the dependence axis of
$g(\lfloor {\mbt} \rfloor)$, $\mbt \in \R^2$.
}
\end{defn}

\begin{prop} \label{depaxis} Let $g \colon \Z^2 \to \R $ satisfy
\begin{equation}\label{gcoefL}
g(\mbt) = \rho(B\mbt)^{-1} (L(B \mbt) + o(1)), \quad |\mbt| \to \infty,
\end{equation}
where $B = (b_{ij})_{i,j=1,2}$ is a $2\times 2 $ nondegenerate  matrix, $\rho(\mbt) := |t_1|^{q_1} + |t_2|^{q_2}$, $\mbt \in \R^2$, with $q_i >0$, $i=1,2$, and $L : \R_0^2 \to \R$ satisfies Assumption B.
In addition,
\smallskip

\noi (i) let $q_1 < q_2 $ and $|L_+(1)|  =  |L_-(1)| >0$, $|L_+(-1)|  = |L_-(-1)| > 0$. Then the dependence axis of $g$ is $\{ \mbt \in \R^2 : {\mbb}_2\cdot {\mbt} = 0\}$ with $\mbb_2 =  (b_{21},b_{22})^\top$;

\smallskip

\noi (ii) let $q_1 > q_2 $ and $|L_+ (0)|>0$, $|L_-(0)| > 0$. Then the dependence axis of $g$ is $\{\mbt \in \R^2 : {\mbb}_1\cdot {\mbt} = 0 \}$ with $\mbb_1 =  (b_{11},b_{12})^\top$.

 \end{prop}

\begin{proof}
It suffices to show part (i) only since (ii) is analogous.
Below we prove that
\begin{align}\label{C1}
\lim_{|{\mbt}|\to \infty, \,  {\mbb}_2\cdot {\mbt} = 0} |{\mbt}|^{q_1} \,
|g(\lfloor {\mbt}\rfloor)| &= \frac{|\mbb_2|^{q_1}}{|\operatorname{det}(B)|^{q_1}}
\begin{cases}L_\pm(1), &{\mbb}_1\cdot {\mbt} \to + \infty, \\
L_\pm(-1), &{\mbb}_1\cdot {\mbt} \to - \infty,
\end{cases}\\
\limsup_{|{\mbt}|\to \infty, \,  {\mbc} \cdot {\mbt} = 0} |{\mbt}|^{q_2}\, |g(\lfloor {\mbt}\rfloor)| &< \infty, \quad \forall
{\mbc} \in \R_0^2, \ b_{21} c_2  \ne c_1 b_{22}. \label{C2}
\end{align}
Note \eqref{C1} implies  $\lim_{|{\mbt}| \to \infty,\, {\mbb}_2\cdot {\mbt} = 0} \frac{\log (1/|g(\lfloor {\mbt}\rfloor)|)}{\log |{\mbt}|} = q_1 $
while \eqref{C2} implies $\liminf_{|{\mbt}| \to \infty, \, {\mbc}\cdot {\mbt} = 0} \frac{\log (1/|g(\lfloor {\mbt}\rfloor )|)}{\log |{\mbt}|}
\ge q_2 $, hence the statement of the proposition.

Let us prove \eqref{C1}. We have 
\begin{equation*}
\frac{{\mbb}_1 \cdot {\mbt }}
{\rho( B \lfloor {\mbt} \rfloor )^{1/q_1}}
= \frac{\operatorname{sign}({\mbb}_1 \cdot {\mbt }) }{ (| {\mbb}_1 \cdot
\lfloor {\mbt } \rfloor/|{\mbb}_1 \cdot {\mbt }| |^{q_1}
+ |{\mbb}_2 \cdot \lfloor {\mbt }\rfloor|^{q_2}/ |{\mbb}_1 \cdot {\mbt }|^{q_1} )^{1/q_1} } \
\to\ \pm 1 \quad \text{as} \quad  {\mbb}_1 \cdot {\mbt } \to \pm \infty
\end{equation*}
since $ |{\mbb}_2 \cdot \lfloor {\mbt }\rfloor|   = O(1) $ on  ${\mbb}_2\cdot {\mbt} = 0$.  In a similar way,
$\lim_{|{\mbt}|\to \infty, \,  {\mbb}_2\cdot {\mbt} = 0} |{\mbt}|^{q_1} \rho(B \lfloor {\mbt} \rfloor)^{-1}= (\frac{|b_{21}| + |b_{22}|}{|\operatorname{det}(B)|})^{q_1}
$. Whence, \eqref{C1} follows by the asymptotic form
of $g$ and the assumption of the continuity of $L_\pm$.

Consider \eqref{C2}. In view of \eqref{gcoefL} and the boundedness of $L_\pm $ it suffices to show
\eqref{C2} for  $\rho(B{\mbt})^{-1}$ in place of $g({\mbt})$, $\mbt \in \Z^2$.
Then 
$|{\mbt}|^{q_2} \rho(B \lfloor {\mbt} \rfloor)^{-1} =
( \frac{| {\mbb}_1 \cdot
\lfloor {\mbt } \rfloor |^{q_1}}{|{\mbt }|^{q_2}}
+ \frac{|{\mbb}_2 \cdot \lfloor {\mbt }\rfloor|^{q_2}}{|{\mbt }|^{q_2}} )^{-1}$, where
$\frac{| {\mbb}_1 \cdot
\lfloor {\mbt } \rfloor |^{q_1}}{|{\mbt }|^{q_2}} \to 0$ and
$\frac{|{\mbb}_2 \cdot \lfloor {\mbt }\rfloor|}{|{\mbt }|} \to \frac{|b_{21} c_2 - b_{22} c_1|}{|c_1| + |c_2|} > 0$,
proving \eqref{C2}.
\end{proof}

Below, we show that the dependence axis is preserved under `discrete' convolution
$[g_1 \star g_2] (\mbt) := \sum_{\mbu \in \Z^2} g_1(\mbu) $  $ g_2(\mbu + \mbt)$, $\mbt \in \Z^2$,
of two functions $g_i: \Z^2 \to \R$, $i=1,2$.

\begin{prop} \label{conaxis} For $i=1,2$, let $g_i \colon \Z^2 \to \R$ satisfy
\begin{equation}
g_i(\mbt) = \rho(B\mbt)^{-1} (L_i(B \mbt) + o(1)), \quad |\mbt| \to \infty,
\end{equation}
where $B$ is a $2\times 2$ nondegenerate matrix, $\rho$ with $Q = q_1^{-1} + q_2^{-1} \in (1,2)$ and $L_i$ are functions as in Assumption~B. For $i=1,2$, let $a_{\infty,i} (\mbt) :=  \rho(\mbt)^{-1} L_i(\mbt)$, $\mbt \in \R^2_0$.
Then
\begin{equation}\label{gg}
[g_1 \star g_2](\mbt) = \tilde \rho(B \mbt)^{-1} (\tilde L ( B \mbt) + o(1)), \quad |\mbt| \to \infty,
\end{equation}
where $\tilde \rho (\mbt) := |t_1|^{\tilde q_1} + |t_2|^{\tilde q_2}$, $\mbt \in \R^2$, with $\tilde q_i := q_i (2-Q)$, $i=1,2$, and
\begin{equation}\label{tildeL}
\tilde L(\mbt) := |\operatorname{det}(B)|^{-1} (a_{\infty,1} \star a_{\infty,2} ) ( t_1/\tilde \rho (\mbt)^{1/\tilde q_1}, t_2/\tilde \rho (\mbt)^{1/\tilde q_2} ),
\quad \mbt \in \R^2_0,
\end{equation}
is a generalized invariant function in the sense of Definition \ref{defhomo} (ii) (with
$q_i$ replaced by $\tilde q_i$, $i=1,2$). Moreover, if $L_1 = L_2 \ge 0$ then $\tilde L$ is strictly positive. 
\end{prop}

\begin{proof}
We follow the proof in  (\cite{pils2017}, Prop.\ 5.1 (iii)). For $\mbt \in \Z^2$, split every $g_i(\mbt)$ as a sum of $g^1_i(\mbt) :=  g_i(\mbt) - g^0_i(\mbt) $ and
$g^0_i(\mbt) := (1 \vee \rho(B\mbt))^{-1} L_i ( B \mbt)$ using the convention $g_i^0(\boldsymbol{0}) = L_i (\boldsymbol{0}) := 0$. Then $[g_1 \star g_2](\mbt) = \sum_{k,j=0}^1 [g^k_1 \star g^j_2](\mbt)$
and \eqref{gg} follows from
\begin{equation}\label{gg0}
\lim_{|\mbt|\to \infty} \big| \tilde \rho(B \mbt) [g^0_1 \star g^0_2](\mbt) - \tilde L ( B\mbt ) \big| = 0
\end{equation}
and
\begin{equation}\label{gg1}
\tilde \rho(B \mbt) [g^k_1 \star g^j_2](\mbt)= o(1), \quad |\mbt| \to \infty, \quad (k,j) \ne (0,0).
\end{equation}
To  prove \eqref{gg0} we write the `discrete' convolution  as integral
$[g^0_1 \star g^0_2](\mbt) = \int_{\R^2} g^0_1(\lceil \mbu \rceil ) g^0_2 (\lceil \mbu \rceil +  \mbt) \d \mbu
$, where we change a variable: $\mbu 
\to B^{-1} R_{\tilde \varrho}  \mbu$ with
$$
\mbt' := B \mbt, \quad \tilde \varrho :=
\tilde \rho(\mbt'), \quad R_{\tilde \varrho} := \operatorname{diag}
( \tilde \varrho^{1/\tilde q_1}, \tilde \varrho^{1/\tilde q_2}).
$$
Then with $\tilde Q:=\tilde q_1^{-1} + \tilde q_2^{-1}$ we have
\begin{align*}
\tilde \rho(B \mbt) [g^0_1 \star  g^0_2](\mbt)
&=|\operatorname{det}(B)|^{-1} \tilde \varrho^{1+\tilde Q}
\int_{\R^2} \frac{ L_1 ( B \lceil B^{-1} R_{\tilde \varrho} \mbu \rceil ) } {\rho ( B \lceil B^{-1} R_{\tilde \varrho} \mbu \rceil ) \vee 1}
\cdot \frac{L_2 ( B \lceil  B^{-1} R_{\tilde \varrho} \mbu \rceil + \mbt' )}
{\rho ( B \lceil B^{-1} R_{\tilde \varrho} \mbu \rceil + \mbt') \vee 1}
 \d \mbu \\
&= |\operatorname{det}(B)|^{-1} \int_{\R^2} g_{\tilde \varrho, \mbz} (\mbu) \d \mbu, \quad \mbz = R_{\tilde \varrho}^{-1} \mbt',
\end{align*}
where for all $\tilde \rho>0$, $\mbz \in \R^2$ such that $\tilde \rho (\mbz) = 1$, $\mbu \in \R^2$,
\begin{equation*}
g_{\tilde \varrho, \mbz} (\mbu)
:=\frac{ L_1 ( R^{-1}_{\tilde \varrho} B \lceil B^{-1} R_{\tilde \varrho} \mbu \rceil ) }
{\rho ( R^{-1}_{\tilde \varrho} B \lceil B^{-1} R_{\tilde \varrho} \mbu \rceil ) \vee \tilde \varrho^{-q_1/\tilde q_1}}
\cdot \frac{L_2 ( R^{-1}_{\tilde \varrho} B \lceil  B^{-1} R_{\tilde \varrho} \mbu \rceil +   \mbz )}
{\rho ( R^{-1}_{\tilde \varrho} B \lceil B^{-1} R_{\tilde \varrho} \mbu \rceil +  \mbz) \vee \tilde \varrho^{-q_1/\tilde q_1} }
\end{equation*}
and we used generalized homogeneous and generalized invariance properties of $\rho $ and $L_i$, $i=1,2$, and
the facts that $q_1/\tilde q_1 = q_2/\tilde q_2$, $1 + \tilde Q = 2 q_1/\tilde q_1 = 2/(2-Q)$. Whence using
continuity of  $\rho $ and $L_i$, $i=1,2$, it follows that $g_{\tilde \varrho, \mbz} (\mbu) -
a_{\infty,1} (\mbu) a_{\infty,2} (\mbu + \mbz) \to  0$ as $ \tilde \varrho \to \infty $ or $|\mbt| \to \infty $
for all $\mbu\in \R^2_0$, $\mbu + \mbz \in \R^2_0$.
Then similarly as in (\cite{pils2017}, (7.8)) we conclude that
$\sup_{\mbz \in \R^2:  \rho(\mbz) =1  } |\int_{\R^2} g_{\tilde \varrho, \mbz} (\mbu) \d \mbu -  (a_{\infty,1} \star a_{\infty,2})(\mbz)| \to 0$,
$\tilde \varrho \to \infty$, and \eqref{gg0} holds. The remaining details including the proof of \eqref{gg1}
are similar to those in \cite{pils2017}. Proposition \ref{conaxis} is proved.
\end{proof}

\begin{cor}\label{cordep}
Let $X$ be a linear RF on $\Z^2 $ satisfying Assumptions A, B and having a covariance function
$r_X (\mbt) := \E X(\boldsymbol{0})X(\boldsymbol{t})   = [b \star b ](\mbt)$, $\mbt \in \Z^2$. Then
\begin{eqnarray} \label{covX}
&r_X(\mbt) = \tilde \rho(B \mbt)^{-1} (\tilde L(B \mbt) + o(1)), \quad |\mbt| \to \infty,
\end{eqnarray}
where $\tilde \rho $, $\tilde L$ are as in \eqref{gg}, \eqref{tildeL} (with $a_{\infty,1} =  a_{\infty,2} = a_\infty $ of \eqref{ainfty}).
Particularly, if $q_1 \ne q_2$ and $L_\pm $ satisfy the conditions in Proposition \ref{depaxis}, the dependence axes
of the covariance function $r_X$ in \eqref{covX} and the moving-average coefficients $b$  in \eqref{bcoefL} coincide.
\end{cor}

\section*{Acknowledgments}

The authors thank Shanghai New York University for hosting their visits in April--May, 2019 during which this work was initiated and partially 
completed. 
Vytaut{\.e} Pilipauskait{\.e} acknowledges the financial support from the project ``Ambit fields: probabilistic properties and statistical inference'' funded by Villum Fonden. Also, Vytaut{\.e} Pilipauskait{\.e} gratefully acknowledges financial support of ERC Consolidator Grant 815703 ``STAMFORD: Statistical Methods for High Dimensional Diffusions''.


\bigskip

\bigskip

{\footnotesize


\begin{thebibliography}{99}





\bibitem{ber2013} Beran, J., Feng, Y., Gosh, S. and Kulik, R. (2013)
{\em Long-memory processes: Probabilistic properties and statistical methods.} Springer, New York.   





\bibitem{dam2017}  Damarackas, J. and Paulauskas, V. (2017)  Spectral covariance and limit theorems for random
fields with infinite variance. J. Multiv. Anal. 153, 156--175.


\bibitem{dam2019}  Damarackas, J. and Paulauskas, V. (2019)  Some remarks on scaling transition in limit
theorems for random fields. Preprint. Available at {\tt arXiv:1903.09399 [math.PR]}.




\bibitem{dav1970} Davydov, Y.A. (1970) The invariance principle for stationary
processes. Theor. Probab. Appl. 15, 487--498.  


\bibitem {dob1979} Dobrushin, R.L. (1979) Gaussian and their subordinated
self-similar random generalized fields.  Ann. Probab. 7,
1--28.




\bibitem {dobmaj1979} Dobrushin, R.L. and Major, P. (1979)
Non-central limit theorems for non-linear functionals of
Gaussian fields. Probab. Th. Rel. Fields
50, 27--52.  


\bibitem{dou2003}  Doukhan, P., Oppenheim, G.  and  Taqqu, M.S. (Eds.) (2003)
{\em Theory and Applications of Long-Range Dependence}.
Birkh\"auser, Boston.














\bibitem{book2012} Giraitis, L., Koul, H.L. and Surgailis, D. (2012)  {\em Large Sample Inference for
Long Memory Processes.}  Imperial College Press,  London.





\bibitem{han1972} Hankey, A. and Stanley, H.E.  (1972)
Systematic application of generalized homogeneous functions to static
scaling, dynamic scaling, and universality. Phys. Review B 6,
3515--3542.






\bibitem{lah2016} Lahiri, S.N. and Robinson, P.M. (2016)
Central limit theorems for long range dependent spatial linear processes.
Bernoulli 22, 345--375.



\bibitem{leo1999} Leonenko, N.N. (1999) {\em
Random Fields with Singular Spectrum}. Kluwer, Dordrecht. 






\bibitem{pils2014} Pilipauskait\.e, V. and Surgailis, D. (2014) Joint temporal and contemporaneous aggregation
of random-coefficient  AR(1) processes. Stochastic Process. Appl. 124, 1011--1035.



\bibitem{pils2016} Pilipauskait\.e, V.  and Surgailis, D. (2016) Anisotropic scaling of random grain model  with application
to network traffic. J. Appl. Probab. 53, 857--879.


\bibitem{pils2017} Pilipauskait\.e, V.  and Surgailis, D. (2017) Scaling  transition
for nonlinear random fields with long-range dependence. Stochastic Process. Appl.
127,  2751--2779.


\bibitem{pra1960} Pratt, J.W. (1960) On interchanging limits and integrals. Ann. Math. Statist.
31, 74--77.


\bibitem{ps2015} Puplinskait\.e, D.  and Surgailis, D. (2015)
Scaling  transition for long-range dependent Gaussian random fields.
Stochastic Process. Appl. 125, 2256--2271.



\bibitem{ps2016} Puplinskait\.e, D.  and Surgailis, D. (2016)
Aggregation of autoregressive random
fields and anisotropic long-range dependence.  Bernoulli 22, 2401--2441.



\bibitem{sam1994} Samorodnitsky, G. and Taqqu, M.S. (1994) {\it Stable
Non-Gaussian Random Processes.}  Chapman and Hall, London.










\bibitem {sur2019a} Surgailis, D. (2019)
Anisotropic scaling limits of long-range dependent
linear random fields on ${\Z}^3$.
J. Math. Anal. Appl. 472, 328--351.

\bibitem {sur2019b} Surgailis, D. (2019)
Scaling transition and edge effects for negatively dependent
linear random fields on ${\Z}^2$. Preprint. Available at {\tt arXiv:1904.05134 [math.PR]}.

\bibitem{sur2019c} Surgailis, D. (2019) Anisotropic scaling limits of long-range dependent
linear random fields. Lithuanian Math. J. 59, 595--615.


\end{thebibliography}

}
\end{document}